\documentclass[11pt]{article}
\usepackage[utf8]{inputenc}
\usepackage{amsmath,amsfonts, bm,amsthm}
\usepackage{amssymb,enumitem,mathtools,mathrsfs}
\usepackage{multirow}
\usepackage{graphicx}
\usepackage{url}
\usepackage{appendix}
\usepackage{hyperref}

\makeatletter
\newcommand*{\transpose}{%
  {\mathpalette\@transpose{}}%
}
\newcommand*{\@transpose}[2]{%
  \raisebox{\depth}{$\m@th#1\intercal$}%
}

\usepackage{lipsum}

\let\OLDthebibliography\thebibliography
\renewcommand\thebibliography[1]{
  \OLDthebibliography{#1}
  \setlength{\parskip}{0pt}
  \setlength{\itemsep}{0pt plus 0.3ex}
}

\usepackage{array,etoolbox}
\preto\tabular{\setcounter{magicrownumbers}{0}}
\newcounter{magicrownumbers}

\preto\tabular{\setcounter{magicrownumbers2}{0}}
\newcounter{magicrownumbers2}

\newtheorem{theorem}{Theorem}[section]
\newtheorem{lemma}[theorem]{Lemma}
\newtheorem{proposition}[theorem]{Proposition}
\newtheorem{corollary}[theorem]{Corollary}
\theoremstyle{definition}
\newtheorem{definition}[theorem]{Definition}
\newtheorem{example}[theorem]{Example}

\theoremstyle{remark}
\newtheorem{remark}[theorem]{Remark}

\newtheorem{question}[theorem]{Question}
\newtheorem{problem}[theorem]{Problem}

\title{Neumaier graphs from cyclotomy with small coherent rank}
\author{Gary Greaves 
  \thanks{School of Physical and Mathematical Sciences, 
  Nanyang Technological University, 
  21 Nanyang Link, Singapore 637371,
 {\tt{gary@ntu.edu.sg}}
 }
 \and
 Zhao Kuang Tan
  \thanks{School of Physical and Mathematical Sciences, 
  Nanyang Technological University, 
  21 Nanyang Link, Singapore 637371,
 \tt{zhaokuan001@e.ntu.edu.sg}}
}
\date{}

\begin{document}

\maketitle
\begin{abstract} 
Using cyclotomy, we construct a new infinite family of Neumaier graphs that includes infinitely many strongly regular graphs. 
Notably, this family conjecturally contains infinitely many graphs with coherent rank~$6$. 
Our construction also provides the first known examples that answer a question posed by Evans, Goryainov, and Panasenko regarding the existence of Neumaier graphs whose nexus is not a power of~$2$. 
In addition, we show that a construction of Greaves and Koolen yields an infinite family of Neumaier graphs with coherent rank~$6$. \end{abstract}

\section{Introduction}

A graph $\Gamma$ of order $v$ is called $k$-\textbf{regular} if each vertex belongs to precisely $k$ edges, a $k$-regular graph is called $(v,k,\lambda)$-\textbf{edge-regular} if each edge belongs to precisely $\lambda$ triangles.
A clique $C$ in $\Gamma$ is called $e$-\textbf{regular} if each vertex outside of $C$ is adjacent to precisely $e > 0$ vertices in $C$. 
A non-complete edge-regular graph is called a \textbf{Neumaier graph} if it has a regular clique.
The study of Neumaier graphs was initiated by Neumaier in 1981.

At the time of his seminal paper~\cite{Neumaier1981}, all known examples of Neumaier graphs were \textbf{strongly regular}, i.e., edge-regular with the additional property that there exists $\mu$ such that every pair of non-adjacent vertices has exactly $\mu$ common neighbors.
This led Neumaier to pose a fundamental question~\cite[Page 248]{Neumaier1981}: are all Neumaier graphs strongly regular?
This question remained open until 2018, when Greaves and Koolen~\cite{GreavesKoolen1} constructed an infinite family of Neumaier graphs that are not strongly regular, thereby answering the question in the negative. A Neumaier graph that is not strongly regular is referred to as a \textbf{strictly Neumaier graph}. Since 2018, several infinite families of strictly Neumaier graphs have been identified~\cite{abiad2023infinite,abiad2021neumaier, evans2021general,EGP19, GreavesKoolen2}.
The \textit{coherent closure} (also known as the \textit{Weisfeiler–Leman closure}) of a graph was originally introduced by Weisfeiler and Leman~\cite{weisfeiler1968reduction} as a tool for distinguishing non-isomorphic graphs. 
For a graph~$\Gamma$, its coherent closure is the minimal coherent algebra containing the adjacency matrix~$A(\Gamma)$. 
This concept is often studied implicitly in the context of distance-regular graphs, where the coherent closure coincides with the Bose–Mesner algebra~\cite{BCN}. 
The study of Schur rings, a special case of coherent algebras, has been applied to solve isomorphism problems for certain Cayley graphs~\cite{muzychuk1999isomorphism}. 
More recently, the coherent closure of a graph has played a central role in the study of Deza graphs~\cite{deza,CHURIKOV22} and graphs with three distinct eigenvalues~\cite{GreavesYip}. 
The \textit{coherent rank} of a graph, defined formally in Section~\ref{sec:cr}, is the rank of its coherent closure. 
Among all non-complete, non-empty graphs, strongly regular graphs attain the smallest possible coherent rank. 
A result of Abiad et al.~\cite{abiad2021neumaier} implies that Neumaier graphs with coherent rank $4$ do not exist. 
However, Neumaier graphs with coherent rank~$6$ do exist. 
In fact, we prove (see Theorem~\ref{thm:cr6inf}) that the second infinite family of strictly Neumaier graphs, described in~\cite{GreavesKoolen2}, contains infinitely many such graphs.

The first infinite family of strictly Neumaier graphs, introduced in~\cite{GreavesKoolen1}, contains only two graphs of coherent rank~$6$: specifically, its smallest members with $28$ and $72$ vertices. 
All larger graphs in this family have coherent rank strictly greater than~$6$. 
Motivated by this, we construct a new infinite family of Neumaier graphs (see Theorem~\ref{thm:main}) that conjecturally contains infinitely many strictly Neumaier graphs with coherent rank~$6$.

Our new construction is an infinite family of Cayley graphs over products of finite fields, with connection sets derived from cyclotomic classes. Naturally, certain known infinite families of strongly regular graphs arise as subfamilies of this construction (see Remark~\ref{rem:srgOA} and Remark~\ref{rem:uniformsrg}).

We say that a Neumaier graph $\Gamma$ has \textbf{parameters} $(v,k,\lambda; e, s)$ if it is $(v,k,\lambda)$-edge-regular and has an $e$-regular clique of order $s$.
By \cite[Theorem 1.1]{Neumaier1981}, the regular cliques of $\Gamma$ are precisely its maximal cliques, which are each $e$-regular and of order $s$. 
The parameter $e$ is called the \textbf{nexus} of $\Gamma$.
The vast majority of currently known constructions of Neumaier graphs have nexus $e=1$~\cite{GreavesKoolen1,GreavesKoolen2,abiad2023infinite, evans2021general}.
Evans et al.~\cite{EGP19} produced the first examples of Neumaier graphs that have nexus $e > 1$. 
However, the nexus of each graph in their families is a power of $2$.
Evans, Goryainov, and Panasenko~\cite{EGP19} therefore asked if there exist Neumaier graphs whose nexus is not a power of $2$.
Our new construction contains examples of Neumaier graphs that answer their question in the affirmative.

The paper is organised as follows.
In Section~\ref{sec:cr}, we introduce the coherent closure of a graph.
Afterwards, we show that the family of Neumaier graphs in \cite{GreavesKoolen2} contain an infinite subfamily of graphs that have coherent rank $6$.
In Section~\ref{sec:newfam}, we introduce a new family of Cayley graphs and provide a sufficient condition which guarantees that these graphs are Neumaier graphs.
We show that our new family of Cayley graphs contains infinitely many strongly regular graphs, which have previously been discovered independently.
In Section~\ref{sec:6beyond}, we restrict our consideration to the graphs of our new construction that have coherent ranks 6 and 7.  
We leverage closed-form expressions for cyclotomic numbers of low order to obtain simple sufficient conditions for the existence of Neumaier graphs having coherent ranks 6 and 7.
We conclude the paper with a conjecture about the sum of products of cyclotomic numbers, which may be of independent interest.
\section{Coherent closure of a graph}
\label{sec:cr}
In this section, we define the coherent closure and the coherent rank of a graph.

\subsection{Coherent algebras and configurations}

Throughout, we denote by $I_n$, $J_n$ and $\mathbf 1_n$, the $n \times n$ identity matrix, the $n \times n$ all-ones matrix, and the $n \times 1$ all-ones matrix (or vector), respectively.
When the order of these matrices is clear from context, we merely write $I$, $J$, and $\mathbf 1$ respectively.

Let $X$ be a finite set and let $R = \{R_0,\dots,R_{\mathnormal r-1}\}$ be a set of binary relations on $X$.
For each $R_i$, the corresponding \textbf{relation matrix} $\mathsf A_i \in \operatorname{Mat}_{X}(\{0,1\})$ is defined such that its $(x,y)$ entry $[A_i]_{xy}$ is $1$ if $(x,y) \in R_i$ and $0$ otherwise.
Suppose that 
\begin{enumerate}
    \item[(CC1)] $\displaystyle \sum_{i=0}^{\mathnormal r-1} \mathsf A_i = J$;
    \item[(CC2)] For each $i \in \{0,\dots,{\mathnormal r-1}\}$, there exists $j \in \{0,\dots,{\mathnormal r-1}\}$ such that $\mathsf A_i^\transpose = \mathsf A_j$;
    \item[(CC3)] There exists a subset $\mathfrak D \subset \{0,\dots,{\mathnormal r-1}\}$ such that $\displaystyle \sum_{i \in \mathfrak D} \mathsf A_i = I$;
    \item[(CC4)] $\displaystyle \mathsf A_i \mathsf A_j = \sum_{k=0}^{\mathnormal r-1} p_{i,j}^k \mathsf A_k$, for each $i,j \in \{0,\dots,{\mathnormal r-1}\}$.
\end{enumerate}
Then $({X},{R})$ is called a \textbf{coherent configuration} of \textbf{rank} ${\mathnormal r}$.
A coherent configuration is called \textbf{homogeneous} if $|\mathfrak D| = 1$, \textbf{symmetric} if $\mathsf A_i  = \mathsf A_i^\transpose$ for each $i \in \{0,\dots,r-1\}$, and \textbf{commutative} if $\mathsf A_i\mathsf A_j  = \mathsf A_j\mathsf A_i$ for each $i,j \in \{0,\dots,r-1\}$.
If a coherent configuration is commutative, then it must be homogeneous, and if it is symmetric, then it must be commutative~\cite[3.8]{Higman1987}.
A homogeneous coherent configuration is called an \textbf{association scheme}.
A \textbf{coherent algebra} is a matrix algebra $\mathcal A \subset \operatorname{Mat}_{X}(\mathbb C)$ that satisfies the following axioms.
\begin{itemize}
    \item[(A1)] $I, J \in \mathcal A$;
    \item[(A2)] $M^\transpose \in \mathcal A$ for each $M \in \mathcal A$;
    \item[(A3)] $MN \in \mathcal A$ and $M \circ N \in \mathcal A$ for each $M,N \in \mathcal A$, where $\circ$ denotes the entrywise product.
\end{itemize}
Each coherent algebra $\mathcal A$ has a unique basis of $\{0,1\}$-matrices $\{\mathsf  A_0,\dots,\mathsf A_{\mathnormal r-1}\}$ that corresponds to a coherent configuration $({X},{R}(\mathcal A))$ where ${R}(\mathcal A) = \{R_0,\dots,R_{r-1}$\} with $(x,y) \in R_i$ if and only if the $(x,y)$-entry $[\mathsf A_i]_{xy} = 1$ for each $i \in \{0,\dots,r-1\}$.
In a slight abuse of language, we refer to the set of $\{0,1\}$-matrices $\{\mathsf  A_0,\dots,\mathsf A_{\mathnormal r-1}\}$ as the \textbf{underlying coherent configuration} of $\mathcal A$.
A coherent algebra $\mathcal A$ is called \textbf{homogeneous} (resp.\ \textbf{symmetric}) if its underlying coherent configuration is homogeneous (resp.\ symmetric).

The following lemma, which is referred to as the Schur-Wielandt principle, is employed below to obtain lower bounds for the rank of a coherent algebra.

\begin{lemma}[{\cite[Theorem 2.3.10]{cc}}]
    \label{lem:wielandtcc}
    Let $\mathcal A$ be a coherent algebra and let $A \in \mathcal A$.
    For $b \in \mathbb C$, define the matrix $B$ such that
    \[
    [B]_{xy} = \begin{cases}
        1, & \text{ if $[A]_{xy} = b$;} \\
        0, & \text{otherwise.}
    \end{cases}
    \]
    Then $B \in \mathcal A$.
\end{lemma}
\subsection{Coherent closure and coherent rank}

Clearly, the intersection of any two coherent algebras is itself a coherent algebra.
We can thus define the \textbf{coherent closure} $\mathcal {W}(\Gamma)$ of a graph $\Gamma$ to be the minimal coherent algebra that contains the adjacency matrix $A(\Gamma)$ of $\Gamma$.
We denote by $\mathcal {C}(\Gamma) = \{ \mathsf A_0,\dots,\mathsf A_{\mathnormal r-1} \}$ the underlying coherent configuration of $\mathcal {W}(\Gamma)$.

Define the \textbf{coherent rank} of $\Gamma$ to be the rank of its coherent closure $\mathcal {W}(\Gamma)$, which we denote by $\operatorname{rk} \mathcal W(\Gamma)$.
The coherent rank of a complete or empty graph on at least $2$ vertices is $2$ and the coherent rank of a strongly regular graph is $3$.

\begin{lemma}
\label{lem:rankeig}
    Let $\Gamma$ be a graph such that $A = A(\Gamma)$ has precisely $r$ distinct eigenvalues.
    Then $\operatorname{rk} \mathcal W(\Gamma) \geqslant r$.
    In the case of equality, we have $\mathcal W(\Gamma) = \langle I, A, \dots, A^{r-1} \rangle$.
\end{lemma}
\begin{proof}
    The matrices $I, A, A^2 ,\dots, A^{r-1}$ form a linearly independent subset of $\mathcal W(\Gamma)$.
    Hence, the rank of $\mathcal W(\Gamma)$ is at least $r$.
    In the case of equality, the matrices $I, A, A^2 ,\dots, A^{r-1}$ form a basis for $\mathcal W(\Gamma)$.
\end{proof}
Graphs for which we have equality in Lemma~\ref{lem:rankeig} are called \textbf{quotient-polynomial graphs}~\cite{quopol}.
Let $A$ be the adjacency matrix of a quotient-polynomial graph with precisely $r$ distinct eigenvalues.
Then $\mathcal {W}(\Gamma)  = \langle I, A, \dots, A^{r-1} \rangle$ and, consequently, each matrix in $\mathcal C(\Gamma)$ is a polynomial in $A$.
The graphs we construct in Section~\ref{sec:6beyond} provide new examples of quotient-polynomial graphs with coherent ranks 6 and 7.

\subsection{Small coherent rank}

By Lemma~\ref{lem:rankeig}, any graph with coherent rank 4 must have at most four distinct eigenvalues.
Hence, the following theorem is an immediate consequence of \cite[Theorem 3.1]{abiad2021neumaier}, which states that there is no Neumaier graph that has precisely four distinct eigenvalues.

\begin{theorem}
\label{thm:smallrank}
    Let $\Gamma$ be a Neumaier graph with coherent rank at most 4.
    Then $\Gamma$ is strongly regular.
\end{theorem}

The following question is open. 
\begin{question}
\label{prob:5}
    Does there exist a Neumaier graph $\Gamma$ with coherent rank 5?
\end{question}

Abiad et al.~\cite[Corollary 2.33]{abiad2021neumaier} showed the adjacency matrix of a strictly Neumaier graph cannot be a relation matrix of a commutative\footnote{\cite[Corollary 2.33]{abiad2021neumaier} states symmetric association scheme, but \cite[Theorem 2.32]{abiad2021neumaier} implies the adjacency matrix of a strictly Neumaier graph cannot be a class matrix of a commutative association scheme.} association scheme.
We restate their result in the context of the coherent closure, as follows.
\begin{theorem}[cf.~{\cite[Corollary 2.33]{abiad2021neumaier}}]
\label{thm:relationSRG}
    Let $\Gamma$ be a Neumaier graph such that $\mathcal W(\Gamma)$ is commutative.
    Suppose that $A(\Gamma) \in \mathcal C(\Gamma)$.
    Then $\Gamma$ is strongly regular.
\end{theorem}

Theorem~\ref{thm:relationSRG} motivates the following question.

\begin{question}
\label{prob:class}
    Does there exist a strictly Neumaier graph $\Gamma$ such that $A(\Gamma) \in \mathcal C(\Gamma)$?
\end{question}

\begin{theorem}
\label{thm:rank6sym}
    Let $\Gamma$ be a Neumaier graph with coherent rank at most $6$.
    Then the coherent closure $\mathcal W(\Gamma)$ is symmetric.
\end{theorem}
\begin{proof}
    Let $A = A(\Gamma)$.
    If $A$ has less than five distinct eigenvalues, then, by \cite[Theorem 3.1]{abiad2021neumaier}, the graph $\Gamma$ is strongly regular and hence $\mathcal W(\Gamma) = \langle I, A, J-I-A\rangle$ is symmetric.
    If $A$ has at least six eigenvalues, then, by Lemma~\ref{lem:rankeig}, $A$ must have precisely six distinct eigenvalues with $\mathcal W(\Gamma) = \langle I, A, \dots, A^5 \rangle$ being clearly symmetric.
    It remains to treat the case where $A$ has precisely five distinct eigenvalues.
    
    Let $\mathcal C(\Gamma) = \left \{\mathsf A_0,\dots,\mathsf A_5 \right \}$.
    Suppose (for a contradiction) that $\mathcal W(\Gamma) = \langle \mathsf A_0,\dots, \mathsf A_5\rangle$ is not symmetric.
    Since $\mathcal W(\Gamma)$ contains all powers of $A$, we can write $\mathcal W(\Gamma) = \langle I,A,\dots, A^4, B\rangle$, for some matrix $B$.
    The matrix $B$ must be non-symmetric.
    Since $\langle I,A,\dots, A^4\rangle$ is not closed under Hadamard multiplication, there must exist matrices $X, Y \in \langle I,A,\dots, A^4\rangle$ such that $X \circ Y \not \in \langle I,A,\dots, A^4\rangle$.
    Obviously, $X \circ Y$ is symmetric.
    On the other hand, $X \circ Y$ must be in the coset $cB + \langle I,A,\dots, A^4\rangle$ for some scalar $c \ne 0$.
    Thus, $B$ must be symmetric, from which we obtain a contradiction.
\end{proof}

We define the \textbf{support} of $\Gamma$ as the subset of $S \subset \mathcal C(\Gamma)$ such that the adjacency matrix $A(\Gamma)$ is equal to the sum of the elements in $S$.
\begin{proposition}
\label{pro:suppCard}
    Let $\Gamma$ be a Neumaier graph with coherent rank $6$.
    Then the cardinality of the support of $\Gamma$ is $2$ or $3$.
\end{proposition}
\begin{proof}
    By Theorem~\ref{thm:rank6sym}, the coherent closure $\mathcal W(\Gamma)$ is symmetric.
    Hence, by Theorem~\ref{thm:relationSRG}, the support of $\Gamma$ is at least $2$.
    Let $\mathcal C(\Gamma) = \{\mathsf A_0 = I, \mathsf A_1, \dots, \mathsf A_5 \}$.
    Suppose (for a contradiction) that the cardinality of the support of $\Gamma$ is $4$.
    Without loss of generality, we can assume that $A = A(\Gamma) = \mathsf A_1 + \mathsf A_2+\mathsf A_3+\mathsf A_4$.
    Since $\Gamma$ is a Neumaier graph, it is $(v,k,\lambda)$-edge-regular for some non-negative integers $v$, $k$, and $\lambda$.
    Hence,
    $A^2 = kI + \lambda A + \mu \mathsf A_5$ for some integer $\mu$.
    Since $\mathsf A_5 = J-I-A$, we find that $\Gamma$ is strongly regular, a contradiction.
    Lastly, $\Gamma$ cannot have support with cardinality $5$ since $\Gamma$ is not a complete graph.
\end{proof}

Proposition~\ref{pro:suppCard} suggests two types of Neumaier graphs of coherent rank 6, distinguished by the cardinality of their support.
We show below that both types indeed exist.
First, we show that the construction of Greaves and Koolen~\cite{GreavesKoolen2} yields infinitely many Neumaier graphs of coherent rank 6 whose support has cardinality $3$.

\subsection{An infinite family with coherent rank 6}
\label{sec:rank6T}

Let $\Gamma = (X,E)$ be a connected graph.
For $x,y \in X$, denote by $d(x,y)$ the distance from $x$ to $y$.
The \textbf{diameter} of $\Gamma$ is the maximum distance over all pairs of vertices $x,y \in X$.
Suppose $\Gamma$ has diameter $D$.
Then $\Gamma$ is called $a$-\textbf{antipodal} if the relation of being at distance $D$ or distance $0$ is an equivalence relation with equivalence classes having size $a$.
Note that we must have $a \geqslant 2$.
The graph $\Gamma$ is called \textbf{distance-regular} if, for any two vertices $x,y \in X$ with $d(x,y)=k$, the number of vertices at distance $i$ from $x$ and distance $j$ from $y$ depends only on $i$, $j$, and $k$.

Recall the construction of Neumaier graphs in \cite{GreavesKoolen2}.
Let $\Delta$ be an $a$-antipodal distance regular graph of diameter $3$.
Then $\Delta$ is edge-regular with parameters $(v,k,\lambda)$ (say).
Suppose that $\lambda+2$ is a multiple of $a$ and set $t = (\lambda+2)/a$.
Since $\Delta$ is distance regular, its coherent closure $\mathcal W (\Delta) = \langle I, \mathsf A_1, \mathsf A_2, \mathsf A_3 \rangle$, where for each $i \in \{1,2,3\}$, the $(x,y)$ entry of $\mathsf A_i$ is equal to $1$ if and only if $d(x,y) = i$.
Define the graph $\mathfrak K(\Delta)$ to be the graph with adjacency matrix 
\[
A(\mathfrak K(\Delta)) = I_t\otimes (\mathsf A_1 +\mathsf A_3) + (J_t-I_t)\otimes (I+\mathsf A_3),
\]
where the symbol $\otimes$ denotes the Kronecker product.

\begin{theorem}[cf. {\cite[Theorem 2.1]{GreavesKoolen2}}]
\label{thm:GKT}
    Let $\Delta$ be an $a$-antipodal distance regular graph of diameter $3$.
    Suppose that $\Delta$ is $(v,k,\lambda)$-edge-regular such that $a$ is a proper\footnote{The case when $a = \lambda +2$ produces strongly regular graphs~\cite{BSRG}.} divisor of $\lambda + 2$.
    Then $\mathfrak K(\Delta)$ is a strictly Neumaier graph with parameters $(v(\lambda+2)/a,k+\lambda+1,\lambda; 1, \lambda+2)$. 
\end{theorem}

Theorem~\ref{thm:GKT} provides infinitely many Neumaier graphs~\cite{GreavesKoolen2}.
Next, we show that each of these Neumaier graphs has coherent rank 6.

Antipodal distance regular graphs of diameter $3$ have been well-studied, for example, see~\cite{GH}.
Since $\Delta$ is distance regular, together with the property of being $a$-antipodal, one can obtain formulas for the products of $\mathsf A_i \mathsf A_j$ using the \textit{intersection array} of $\Delta$, which is $\{ k, k-\lambda-1, 1; 1, \frac{k-\lambda-1}{a-1}, k \}$.
We refer the reader to \cite[Page 127]{BCN} for the relevant properties of distance regular graphs.
In particular, we have
\begin{align}
\label{eqn:adrg}
    \begin{split}
        \mathsf A_1^2 &= kI + \lambda \mathsf A_1 + \frac{k-\lambda-1}{a-1}\mathsf A_2; \\
        \mathsf A_3^2 &= (a-1)I + (a-2)\mathsf A_3; \\
        \mathsf A_1\mathsf A_3 &= \mathsf A_2; \\
        \mathsf A_2\mathsf A_3 &= (a-1)\mathsf A_1+(a-2)\mathsf A_2.
    \end{split}
\end{align}

Now we show that the Neumaier graphs $\mathfrak K(\Delta)$ have coherent rank 6.

\begin{theorem}
\label{thm:cr6inf}
Let $\Delta$ be an $a$-antipodal distance regular graph of diameter $3$.
    Suppose that $\Delta$ is $(v,k,\lambda)$-edge-regular such that $a$ is a proper divisor of $\lambda + 2$ and suppose that $\mathcal W (\Delta) = \langle I, \mathsf A_1, \mathsf A_2, \mathsf A_3 \rangle$, where $\mathsf A_1 = A(\Delta)$.
    Then
    \[
    \mathcal W (\mathfrak K(\Delta)) =  \left \langle I, I_t\otimes \mathsf A_1,I_t\otimes \mathsf A_2, I_t\otimes \mathsf A_3, (J_t-I_t)\otimes (I+\mathsf A_3), (J_t-I_t)\otimes (\mathsf A_1+\mathsf A_2) \right \rangle.
    \]
\end{theorem}
\begin{proof}
    It is routine to verify that
    \[
    \left \langle I, I_t\otimes \mathsf A_1,I_t\otimes \mathsf A_2, I_t\otimes \mathsf A_3, (J_t-I_t)\otimes (I+\mathsf A_3), (J_t-I_t)\otimes (\mathsf A_1+\mathsf A_2) \right \rangle
    \]
    is a coherent algebra.
    It remains to show that the rank of $\mathcal W = \mathcal W (\mathfrak K(\Delta))$ is at least $6$.
    Let $A = A(\mathfrak K(\Delta))$.
    Using \eqref{eqn:adrg}, we find that
    \[
    \begin{split}
        A^2 =& \  (k+\lambda+1)I + \lambda I\otimes (\mathsf A_1+\mathsf A_3) + \frac{k-\lambda+2a-3}{a-1}I\otimes \mathsf A_2 \\
        & + \lambda(J-I)\otimes (I+\mathsf A_3) +  2(J-I)\otimes (\mathsf A_1+\mathsf A_2),
    \end{split}
    \]
    which clearly belongs to $\mathcal W$.
    By Theorem~\ref{thm:GKT}, the graph $\mathfrak K(\Delta)$ is not strongly regular.
    Hence, $(k-\lambda+2a-3)/(a-1) \ne 2$.
    Furthermore, we can apply Lemma~\ref{lem:wielandtcc} to $(J-I-A) \circ A^2 \in \mathcal W$ to deduce that $I \otimes \mathsf A_2 \in \mathcal W$.
     Using \eqref{eqn:adrg}, we find that
     \[
      \begin{split}
        A(I\otimes \mathsf A_2) =& \ (k -\lambda + a -2) I\otimes \mathsf A_1+ \frac{(a-2)(k+a-2)+\lambda}{a-1} I\otimes \mathsf A_2 \\ 
        & + k I\otimes \mathsf A_3 
        +  (a-1) (J-I)\otimes (\mathsf A_1+\mathsf A_2),
    \end{split}
     \]
     which also belongs to $\mathcal W$.
     Since $a \ne \lambda +2$, by applying Lemma~\ref{lem:wielandtcc} to $A \circ (A(I\otimes \mathsf A_2))$, we find that both $I\otimes \mathsf A_1$ and $I\otimes \mathsf A_3$ belong to $\mathcal W$.
     By applying Lemma~\ref{lem:wielandtcc} to $(J-A-I) \circ (A(I\otimes \mathsf A_2))$, we find that $(J-I)\otimes (\mathsf A_1+A_2)$ belongs to $\mathcal W$.
     It is straightforward now to deduce that $(J-I) \otimes (I+\mathsf A_3) \in \mathcal W$.
\end{proof}

Theorem~\ref{thm:cr6inf} shows that the graphs $\mathfrak K(\Delta)$ are an infinite family of Neumaier graphs that have coherent rank 6 whose support has cardinality 3. 
Evans et al.~\cite{EGP19} generalised Theorem~\ref{thm:GKT}, however, their additional Neumaier graphs have coherent rank greater than 6.

There are only two known examples of Neumaier graphs with coherent rank 6 and support of cardinality 2 in the literature.
These two examples come from a construction in \cite{GreavesKoolen1} with $q = 7$ and $q=13$  (see also Table~\ref{tab:list}, below).
In the remainder of the paper, we exhibit a new family of Neumaier graphs that contains conjecturally infinitely many members having coherent rank 6 and support of cardinality 2.

\section{ A new family of Neumaier graphs from cyclotomy}
\label{sec:newfam}
\subsection{Cyclotomic numbers}

Fix an integer $m \geqslant 2$.
Let $q$ be a prime power congruent to $1$ modulo $m$ and let $\alpha$ be a primitive element for the finite field $\operatorname{GF}(q)$.
Given $a,b \in \mathbb Z$, we define the \textbf{cyclotomic number} $c_m(\alpha;a,b)$ of \textbf{order} $m$ as
\[
c_m(\alpha;a,b) := \left | \left \{ \alpha^k+1 \; : \; k \equiv a \pmod m \right \} \cap \left \{ \alpha^k \; : \; k \equiv b \pmod m  \right \} \right |.
\]

Here, we adopt a slightly unorthodox notation for cyclotomic numbers that includes the primitive element $\alpha$.
We do this because the choice of $\alpha$ will play an important role later on (see Section~\ref{sec:6beyond}).

We will require the following theorem, which outlines some elementary identities for sums of cyclotomic numbers.

\begin{theorem}[{\cite[cf.\ Lemma 3]{Storer1967},\cite[cf.\ Theorem 1.17]{Ding15}}]
\label{thm:cycBasic}
    Let $m \geqslant 2$ be an integer and $q$ be a prime power satisfying $q = 1+nm$ for some $n \in \mathbb N$.
    Let $\alpha$ be a primitive element for $\operatorname{GF}(q)$.
    Then the following equations hold.
    \begin{enumerate}
        \item[(i)] $c_m(q,\alpha;a,b) = \begin{cases}
            c_m(q,\alpha;b,a) & \text{ if $q$ or $n$ is even}; \\
            c_m(q,\alpha;b+m/2,a+m/2) & \text{ if $qn$ is odd}.
        \end{cases}$
        \item[(ii)] $\displaystyle \sum_{a=0}^{m-1}c_m(\alpha;a,b) = \begin{cases}
            n-1 & \text{if $m$ divides $b$;} \\ n & \text{otherwise}.
        \end{cases}$
        \item [(iii)] $\displaystyle \sum_{b=0}^{m-1} c_m(\alpha;a,b) = \begin{cases}
            n-1 & \text{if $a \equiv 0 \pmod m$ and $qn$ is even;} \\
            n-1 & \text{if $a \equiv m/2 \pmod m$ and $qn$ is odd;} \\
            n & \text{otherwise}.
        \end{cases}$
    \end{enumerate}
\end{theorem}

\subsection{Schur rings and Cayley graphs}

Let $G$ be a finite multiplicative group with identity $\mathbf 1_G$.
The group ring $\mathbb Z G$ consists of all formal sums $\sum_{g \in G} c_g g$, where each $c_g \in \mathbb Z$ and the (usual) sum and product of two formal sums are defined as
\begin{align*}
\left (\sum_{g \in G} c_g g \right ) + \left (\sum_{g \in G} c^\prime_g g \right ) &:= \sum_{g \in G} \left (c_g +c_g^\prime \right) g; \\
    \left (\sum_{g \in G} c_g g \right ) \left (\sum_{g \in G} c^\prime_g g \right ) &:= \sum_{g \in G}\sum_{h \in G} \left (c_g c_h^\prime \right) gh.
\end{align*}
For a subset $S \subset G$, define ${S}^{(-1)} := \{ s^{-1} \; : \; s \in S\}$ and
\[
\underline{S} := \begin{cases}
    \displaystyle \sum_{s \in S} s, & \text{ if $S$ is nonempty}; \\
    0 , & \text{ otherwise.}
\end{cases}
\]
A subring $\mathcal S$ of $\mathbb ZG$ is called a \textbf{Schur ring} over the group $G$ if
\begin{itemize}
    \item[(S1)] $\mathcal S$ has a $\mathbb Z$-basis $\underline{ \mathsf B}_0,\dots,\underline{\mathsf B}_{r-1}$, where $\{\mathsf B_0,\dots,\mathsf B_{r-1}\}$ is a partition of $G$ and $\mathsf B_0 = \{\mathbf 1_G\}$;
    \item[(S2)] for each $i \in \{0,1,\dots,r-1\}$ there exists $j \in \{0,1,\dots,r-1\}$ such that $\underline{\mathsf B}_j = \underline{\mathsf B}_i^{(-1)}$.
\end{itemize}
Let $\mathcal S$ be a Schur ring over $G$. 
The basis $\underline{ \mathsf B}_0,\dots,\underline{\mathsf B}_{r-1}$ of (S1) is unique, the corresponding subsets ${\mathsf B}_0,\dots,{\mathsf B}_{r-1}$ of $G$ are called the \textbf{basic sets}
 of $\mathcal S$, and the partition $\{\mathsf B_0,\dots,\mathsf B_{r-1}\}$ is called a \textbf{Schur partition} of $G$.
 Accordingly, we write $\mathcal S = \langle \underline{\mathsf B}_0,\dots,\underline{\mathsf B}_{r-1} \rangle$.
 The \textbf{rank} of $\mathcal S$ is defined as $r$ and denoted by $\operatorname{rk}(\mathcal S)$.
 Given a subset $S \subset G$, we define the \textbf{Cayley digraph} $\operatorname{Cay}(G,S)$ as the digraph with vertex set $G$ and arc set $\{(g,sg) \; : \; g \in G,\, s \in S \}$.
 When $S = S^{(-1)}$ and $\mathbf 1_G \not \in S$ we call $\operatorname{Cay}(G,S)$ an (undirected, simple) \textbf{Cayley graph}.

 \begin{lemma}[{\cite[Lemma 5.1]{deza}}]
 \label{lem:schurCC}
     Let $G$ be a finite group, $S \subseteq G$ with $\mathbf 1_G \not\in S$.
     Suppose that $\mathcal S$ is the minimal Schur ring over $G$ for which $\underline S \in \mathcal S$.
     Then $\operatorname{rk}{\mathcal W}(\operatorname{Cay}(G,S)) = \operatorname{rk}(\mathcal S)$.
 \end{lemma}

 More concretely, suppose that $S \subseteq G$ with $\mathbf 1_G \not\in S$ and $\mathcal S$ is the minimal Schur ring over $G$ that contains $S$.
 Let $\mathsf B_0,\dots,\mathsf B_{r-1}$ be the basic sets of $\mathcal S$.
 Then we have 
 \[
 \mathcal W(\operatorname{Cay}(G,S)) = \langle A(\operatorname{Cay}(G,\mathsf B_0)), \dots, A(\operatorname{Cay}(G,\mathsf B_{r-1})) \rangle.
 \]

    For the proof of Corollary~\ref{cor:rank7} (below), we will require a special case of the Schur-Wielandt principle (see Lemma~\ref{lem:wielandtcc}), which we record as the following lemma.

 \begin{lemma}[{\cite[Proposition 22.1]{weilandt64}, \cite[Proposition 3.2]{muzychuk1999isomorphism}}]
 \label{lem:Wielandt}
    Let $G$ be a finite group, $\mathcal S$ be a Schur ring over $G$, and $\displaystyle \sum_{g \in G} c_g g \in \mathcal S$.
    Then, for any $c \in \mathbb Z$, we have $\underline{\{g \in G \;:\; c_g = c\}} \in \mathcal S$.
\end{lemma}

 \subsection{ The general construction }

 Fix an integer $m \geqslant 2$.
Let $q_1$ and $q_2$ be prime powers congruent to $1$ modulo $m$.
Let $q_1$ and $q_2$ be prime powers satisfying $q_1 = 1+n_1m$ and $q_2 = 1+n_2m$ for some $n_1, n_2 \in \mathbb N$.
Let $\alpha_1$ be a primitive element for $\operatorname{GF}(q_1)$ and $\alpha_2$ be a primitive element for $\operatorname{GF}(q_2)$.
Define $\mathsf C_1(\alpha_1) := \{ (\alpha_1^k,0) \; : \; k \in \{0,\dots,q_1-2\} \}$, $\mathsf C_2(\alpha_2) := \{ (0,\alpha_2^k) \; : \; k \in \{0,\dots,q_2-2\} \}$ and, for each $i \in \{0,\dots,m-1\}$, define the set
\begin{align*}
    \mathsf D_{i}(\alpha_1,\alpha_2) := \left \{ (\alpha_1^{i_1},\alpha_2^{i_2}) \; : \; i_1-i_2 \equiv i \pmod m  \right \}.
\end{align*}
\begin{definition}
    We define $\Gamma_m(\alpha_1, \alpha_2) := \operatorname{Cay}\left (\operatorname{GF}(q_1) \times \operatorname{GF}(q_2), \mathsf C_1(\alpha_1) \cup \mathsf D_{0}(\alpha_1,\alpha_2 ) \right )$.
\end{definition}

It is straightforward to verify that $\underline{C_1(\alpha_1)} = -\underline{C_1(\alpha_1)}$ and $\underline{\mathsf D_{0}(\alpha_1,\alpha_2)} = -\underline{\mathsf D_{0}(\alpha_1,\alpha_2)}$ if and only if $q_1n_1 \equiv q_2n_2 \pmod 2$. 
Hence, we have the following lemma.

\begin{lemma}
\label{lem:undirect}
The graph $\Gamma_m(\alpha_1, \alpha_2)$ is undirected if and only if $q_1n_1 \equiv q_2n_2 \pmod 2$. 
\end{lemma}

Observe that the graph $\Gamma_m(\alpha_1, \alpha_2)$ is a regular graph with $q_1 q_2$ vertices each of degree $(q_1-1)(n_2+1)$.
Define the following sum of products of cyclotomic numbers
\[
\mathcal X^{(m)}_{i,j,k}(\alpha_1,\alpha_2) := \sum_{a = 0}^{m-1}\sum_{b =0}^{m-1} c_m(\alpha_1;a,b)c_m(\alpha_2;a+i-j,b+i-k).
\]
We will use the following obvious identity.
\begin{proposition}
\label{pro:Xeq}
For all $a, b \in \{0,\dots,m-1\}$.
We have $\mathcal X^{(m)}_{0,0,b}(\alpha_1,\alpha_2) = \mathcal X^{(m)}_{a,a,a+b}(\alpha_1,\alpha_2)$.
\end{proposition}

Before introducing our main tool (the following lemma), we first define a pair of auxiliary functions.
\begin{align*}
    \underline{F_{i,j}^{(m)}(\alpha_1,\alpha_2)} &:= n_1n_2\left (\underline{\mathsf C_1(\alpha_1)}+\underline{\mathsf C_2(\alpha_2)}\right ) + \sum_{k=0}^{m-1} \mathcal X^{(m)}_{i,j,k}(\alpha_1,\alpha_2)\underline{\mathsf D_k(\alpha_1,\alpha_2)}; \\
    \underline{G_{i,j}^{(m)}(\alpha_1,\alpha_2)} &:= \underline{F_{i,j}^{(m)}(\alpha_1,\alpha_2)} -n_2\underline{\mathsf C_1(\alpha_1)} - n_1 \underline{\mathsf C_2(\alpha_2)}.
\end{align*}

\begin{lemma}
\label{lem:Spart-rankm+3}
Let $m \geqslant 2$ and $n_1,n_2 \geqslant 1$ be integers. Suppose that $q_1 = 1+mn_1$ and $q_2 = 1+mn_2$ are prime powers.
Let $\alpha_1$ and $\alpha_2$ be primitive elements of $\operatorname{GF}(q_1)$ and $\operatorname{GF}(q_2)$, respectively.
Then 
    \[
    \left \{ \{(0,0)\}, \mathsf C_1(\alpha_1), \mathsf C_2(\alpha_2) \right \} \cup \left \{ \mathsf D_{i}(\alpha_1,\alpha_2) \; : \; i \in \{0,\dots,m-1\}  \right \}
    \]
    is a Schur partition of $\operatorname{GF}(q_1) \times \operatorname{GF}(q_2)$.
    In particular, for $i \in \{1,2\}$ and $j \in \{0,\dots,m-1\}$, 
        \begin{align}
        \label{eqn:641}
        \underline{\mathsf C_i(\alpha_i)}^2 ={}& (q_i-1)\underline{\{(0,0)\}} + (q_i-2)\underline{\mathsf C_i(\alpha_i)}; \\
        \label{eqn:642}
        \underline{\mathsf C_1(\alpha_1)} \, \underline{\mathsf C_2(\alpha_2)} ={}& \sum_{k=0}^{m-1}\underline{\mathsf D_{k}(\alpha_1,\alpha_2)}; \\
        \label{eqn:CD}
        \underline{\mathsf C_i(\alpha_i)}\,\underline{\mathsf D_{j}(\alpha_1,\alpha_2)} ={}& (n_i-1)\underline{\mathsf D_{j}(\alpha_1,\alpha_2)}+n_i\left( \sum_{k \ne i}\underline{\mathsf C_k(\alpha_k)}+ \sum_{k \ne j}\underline{\mathsf D_{k}(\alpha_1,\alpha_2)}\right);
        \end{align}
        and, for each $i,j \in \{0,\dots,m-1\}$, we have
\begin{itemize}
    \item if $q_1n_1$ and $q_2n_2$ have the same parity: 
        \begin{equation}
            \label{eqn:DD1}
            \underline{\mathsf D_{i}(\alpha_1,\alpha_2)}\,\underline{\mathsf D_{j}(\alpha_1,\alpha_2)} = \begin{cases}
                mn_1n_2\underline{\{(0,0)\}} + \underline{G_{i,j}^{(m)}(\alpha_1,\alpha_2)} & \text{ if $i=j$,} \\
                \underline{F_{i,j}^{(m)}(\alpha_1,\alpha_2)} & \text{ otherwise;} 
            \end{cases}
        \end{equation}
        \item if $q_1n_1$ and $q_2n_2$ have different parities: 
        \begin{equation}
            \label{eqn:DD2}
            \underline{\mathsf D_{i}(\alpha_1,\alpha_2)}\,\underline{\mathsf D_{j}(\alpha_1,\alpha_2)} = \begin{cases}
                 \underline{F_{i,j}^{(m)}(\alpha_1,\alpha_2)} + mn_1n_2\underline{\{(0,0)\}} & \text{ if $i=j$,} \\
                 \underline{G_{i,j}^{(m)}(\alpha_1,\alpha_2)} & \text{ if $m \; | \;i - j - m/2$,} \\
                \underline{F_{i,j}^{(m)}(\alpha_1,\alpha_2)} & \text{otherwise.} 
            \end{cases}
        \end{equation}
\end{itemize}
\end{lemma}
\begin{proof}
Equations \eqref{eqn:641} and \eqref{eqn:642} are straightforward to verify.
We leave this to the reader.

Let $i \in \{1,2\}$ and $j \in \{0,\dots,m-1\}$.
First, we consider the multiplicity of $\mathsf C_2(\alpha_2)$ in $\underline{\mathsf C_1(\alpha_1)}\, \underline{\mathsf D_j(\alpha_1,\alpha_2)}$.
The multiplicity of $(0,\alpha_2^{k_2})$ in $\underline{\mathsf C_1(\alpha_1)}\, \underline{\mathsf D_j(\alpha_1,\alpha_2)}$ is the number of elements $((\alpha_1^{i_1},0), (\alpha_1^{j_1},\alpha_2^{j_2})) \in \mathsf C_1(\alpha_1) \times \mathsf D_j(\alpha_1,\alpha_2)$ such that $(\alpha_1^{i_1},0)+(\alpha_1^{j_1},\alpha_2^{j_2}) =(0,\alpha_2^{k_2})$.
Clearly, $j_2 = k_2$ and $i_1 = j_1$ if $q_1$ is even and $i_1 = j_1+(q_1-1)/2$ if $q_1$ is odd.
Since $j_1-j_2 \equiv j \pmod m$, it follows that the multiplicity of $(0,\alpha_2^{k_2})$ in $\underline{\mathsf C_1(\alpha_1)}\, \underline{\mathsf D_j(\alpha_1,\alpha_2)}$ is $(q_1-1)/m$.
In a similar fashion, one can show that the multiplicity of $\mathsf C_k(\alpha_k)$ in $\underline{\mathsf C_i(\alpha_i)}\, \underline{\mathsf D_j(\alpha_1,\alpha_2)}$ is $(q_i-1)/m$ if $k \ne i$ and $0$ otherwise.

Fix $k \in \{0,\dots,m-1\}$.
Next, we consider the multiplicity of elements of $\mathsf D_k(\alpha_1,\alpha_2)$ in $\underline{\mathsf C_i(\alpha_i)}\, \underline{\mathsf D_j(\alpha_1,\alpha_2)}$.
Suppose $(\alpha_1^{k_1},\alpha_2^{k_2}) \in \mathsf D_k(\alpha_1,\alpha_2)$.
Thus, $k_1-k_2 \equiv k \pmod m$. 
The multiplicity of $(\alpha_1^{k_1},\alpha_2^{k_2})$ in $\underline{\mathsf C_1(\alpha_1)}\, \underline{\mathsf D_j(\alpha_1,\alpha_2)}$ is equal to the cardinality of
\[
\left \{ ((\alpha_1^{i_1},0),(\alpha_1^{j_1},\alpha_2^{j_2})) \in \mathsf C_1(\alpha_1) \times \mathsf D_j(\alpha_1,\alpha_2) \; : \; (\alpha_1^{i_1},0)+(\alpha_1^{j_1},\alpha_2^{j_2}) =(\alpha_1^{k_1},\alpha_2^{k_2}) \right \}.
\]
Clearly $j_2 = k_2$, and so we are counting the number of solutions to $\alpha_1^{i_1} + \alpha_1^{j_1} = \alpha_1^{k_1}$, where $k_1-j_1 \equiv k-j \pmod m$ and $i_1-j_1 \in \{0,\dots,q_1-2\}$.
Observe that $c_m(\alpha_1; a,k-j)$ is equal to the number of all such solutions where $i_1-j_1 \equiv a \pmod m$.
Thus, the multiplicity of $(\alpha_1^{k_1},\alpha_2^{k_2})$ in $\underline{\mathsf C_1(\alpha_1)}\, \underline{\mathsf D_j(\alpha_1,\alpha_2)}$ is $\sum_{a=0}^{m-1}c_m(\alpha_1; a,k-j)$.

Similarly, one can show mutatis mutandis that the multiplicity of $(\alpha_1^{k_1},\alpha_2^{k_2})\in \mathsf D_k(\alpha_1,\alpha_2)$ in $\underline{\mathsf C_2(\alpha_2)}\,\underline{\mathsf D_j(\alpha_1,\alpha_2)}$ is $\sum_{a=0}^{m-1}c_m(\alpha_2; a,k-j)$.
Using Theorem~\ref{thm:cycBasic} (ii), we obtain  \eqref{eqn:CD}.

Next, we consider the multiplicity of elements of $\mathsf C_1(\alpha_1)$ in $\underline{\mathsf D_i(\alpha_1,\alpha_2)}\,\underline{\mathsf D_j(\alpha_1,\alpha_2)}$.
The multiplicity of $(\alpha_1^{k_1},0)$ in $\underline{\mathsf D_i(\alpha_1,\alpha_2)}\,\underline{\mathsf D_j(\alpha_1,\alpha_2)}$ is equal to the cardinality of
\[
\left \{ ((\alpha_1^{i_1},\alpha_2^{i_2}),(\alpha_1^{j_1},\alpha_2^{j_2})) \in \mathsf D_i(\alpha_1,\alpha_2) \times \mathsf D_j(\alpha_1,\alpha_2) \; : \; (\alpha_1^{i_1},\alpha_2^{i_2})+(\alpha_1^{j_1},\alpha_2^{j_2}) =(\alpha_1^{k_1},0) \right \}.
\]
Thus, $i_2 = j_2$ if $q_2$ is even and $i_2 = j_2+mn_2/2$ if $q_2$ is odd.
We have $j_1-j_2 \equiv j$ and $i_1-i_2 \equiv i$, which implies $i_1-j_1 \equiv i-j \pmod m$ if $q_2$ is even and $i_1-j_1 \equiv i-j+\frac{mn_2}{2} \pmod m$ if $q_2$ is odd.
For fixed $k_1$, the total number of solutions $(i_1,j_1)$ to $\alpha_1^{i_1}+\alpha_1^{j_1}=\alpha_1^{k_1}$ is $c_m(\alpha_1; i-j,k_1-j_1)$ if $q_2$ is even and $c_m(\alpha_1; i-j+\frac{mn_2}{2},k_1-j_1)$ if $q_2$ is odd.
For each solution $(i_1,j_1)$, there are $n_2$ values of $i_2 \in \{0,\dots,q_2-2\}$ that satisfy $i_1-i_2 \equiv i \pmod m$.
Hence, using Theorem~\ref{thm:cycBasic} (iii), if $q_2$ is even, the multiplicity of $(\alpha_1^{k_1},0)$ in $\underline{\mathsf D_i(\alpha_1,\alpha_2)}\,\underline{\mathsf D_j(\alpha_1,\alpha_2)}$ is 
\[
n_2\sum_{j_1=0}^{m-1} c_m(\alpha_1; i-j,k_1-j_1) = \begin{cases}
    (n_1-1)n_2, & \text{ if $ m \; | \; i -j$ and $q_1n_1$ is even, } \\
    (n_1-1)n_2, & \text{ if $m \; | \; i-j-\frac{m}{2}$ and $q_1n_1$ is odd, } \\
    n_1n_2, & \text{ otherwise;} \\
\end{cases}
\]
and if $q_2$ is odd then the multiplicity of $(\alpha_1^{k_1},0)$ in $\underline{\mathsf D_i(\alpha_1,\alpha_2)}\,\underline{\mathsf D_j(\alpha_1,\alpha_2)}$ is 
\[
n_2\sum_{j_1=0}^{m-1} c_m(\alpha_1; i-j+\frac{mn_2}{2},k_1-j_1) = \begin{cases}
    (n_1-1)n_2, & \text{ if $ m \; | \; i -j+\frac{mn_2}{2}$ and $q_1n_1$ is even, } \\
    (n_1-1)n_2, & \text{ if $m \; | \; i-j+\frac{m(n_2-1)}{2}$ and $q_1n_1$ is odd, } \\
    n_1n_2, & \text{ otherwise.} \\
\end{cases}
\]
Similarly, it follows that if $q_1$ is even, the multiplicity of $(0,\alpha_2^{k_2})$ in $\underline{\mathsf D_i(\alpha_1,\alpha_2)}\,\underline{\mathsf D_j(\alpha_1,\alpha_2)}$ is 
\[
n_1\sum_{j_2=0}^{m-1} c_m(\alpha_2; i-j,k_2-j_2) = \begin{cases}
    n_1(n_2-1), & \text{ if $ m \; | \; i -j$ and $q_2n_2$ is even, } \\
    n_1(n_2-1), & \text{ if $m \; | \; i-j-\frac{m}{2}$ and $q_2n_2$ is odd, } \\
    n_1n_2, & \text{ otherwise;} \\
\end{cases}
\]
and if $q_1$ is odd then the multiplicity of $(0,\alpha_2^{k_2})$ in $\underline{\mathsf D_i(\alpha_1,\alpha_2)}\,\underline{\mathsf D_j(\alpha_1,\alpha_2)}$ is  
\[
n_1\sum_{j_2=0}^{m-1} c_m(\alpha_2; i-j+\frac{mn_1}{2},k_2-j_2) = \begin{cases}
    n_1(n_2-1), & \text{ if $ m \; | \; i -j+\frac{mn_1}{2}$ and $q_2n_2$ is even, } \\
    n_1(n_2-1), & \text{ if $m \; | \; i-j+\frac{m(n_1-1)}{2}$ and $q_2n_2$ is odd, } \\
    n_1n_2, & \text{ otherwise;} \\
\end{cases}
\]

Lastly, we consider the multiplicity of elements of $\mathsf D_k(\alpha_1,\alpha_2)$ in $\underline{\mathsf D_i(\alpha_1,\alpha_2)}\,\underline{\mathsf D_j(\alpha_1,\alpha_2)}$.
Suppose that $(\alpha_1^{k_1},\alpha_2^{k_2}) \in \mathsf D_k(\alpha_1,\alpha_2)$.
The multiplicity of $(\alpha_1^{k_1},\alpha_2^{k_2})$ in $\underline{\mathsf D_i(\alpha_1,\alpha_2)}\,\underline{\mathsf D_j(\alpha_1,\alpha_2)}$ is equal to the cardinality of
\[
\left \{ ((\alpha_1^{i_1},\alpha_2^{i_2}),(\alpha_1^{j_1},\alpha_2^{j_2})) \in \mathsf D_i(\alpha_1,\alpha_2) \times \mathsf D_j(\alpha_1,\alpha_2) \; : \;  (\alpha_1^{i_1},\alpha_2^{i_2})+(\alpha_1^{j_1},\alpha_2^{j_2}) =(\alpha_1^{k_1},\alpha_2^{k_2}) \right \}.
\]
For fixed $a, b \in \{0,\dots,m-1\}$, the total number of solutions $(i_1,i_2,j_1,j_2)$ to the equations $\alpha_1^{i_1}+\alpha_1^{j_1}=\alpha_1^{k_1}$ and $\alpha_2^{i_2}+\alpha_1^{j_2}=\alpha_1^{k_2}$ such that $j_1 - i_1 \equiv a \pmod m $ and $k_1-i_1 \equiv b \pmod m$ is $c_m(\alpha_1; a,b)c_m(\alpha_2; a+i-j,b+i-k)$.
Therefore, the multiplicity of $(\alpha_1^{k_1},\alpha_2^{k_2})$ in $\underline{\mathsf D_i(\alpha_1,\alpha_2)}\,\underline{\mathsf D_j(\alpha_1,\alpha_2)}$ is  $\mathcal X_{i,j,k}^{(m)}(\alpha_1,\alpha_2)$.
Thus, we obtain \eqref{eqn:DD1} and \eqref{eqn:DD2}.
\end{proof}

Now we exhibit a few consequences of Lemma~\ref{lem:Spart-rankm+3}.
The first corollary follows directly from Lemma~\ref{lem:schurCC}.

\begin{corollary}
\label{cor:coherentRanklem}
Let $m \geqslant 2$ and $n_1,n_2 \geqslant 1$ be integers. Suppose that $q_1 = 1+mn_1$ and $q_2 = 1+mn_2$ are prime powers.
Let $\alpha_1$ and $\alpha_2$ be primitive elements of $\operatorname{GF}(q_1)$ and $\operatorname{GF}(q_2)$, respectively.
Then $\Gamma_m(\alpha_1,\alpha_2)$ has coherent rank at most $m+3$.
\end{corollary}

Note that the coherent closure of $\Gamma_m(\alpha_1,\alpha_2)$ may have rank less than $m+3$.
Indeed, we shall see examples below when the coherent rank of $\Gamma_m(\alpha_1,\alpha_2)$ is equal to $3$, which corresponds to the case when $\Gamma_m(\alpha_1,\alpha_2)$ is a strongly regular graph.

Next, we obtain a necessary and sufficient condition for $\Gamma_m(\alpha_1,\alpha_2)$ to be a Neumaier graph, which is the main result of this section.

\begin{theorem}
\label{thm:main}
Let $m \geqslant 2$ and $n_1,n_2 \geqslant 1$ be integers. Suppose that $q_1 = 1+mn_1$ and $q_2 = 1+mn_2$ are prime powers.
Let $\alpha_1$ and $\alpha_2$ be primitive elements of $\operatorname{GF}(q_1)$ and $\operatorname{GF}(q_2)$, respectively.
Then $\Gamma_m(\alpha_1,\alpha_2)$ is a Neumaier graph if and only if $q_1n_1 \equiv q_2n_2 \pmod 2$ and
\begin{equation}
\label{eqn:edgeregular}
    \mathcal X^{(m)}_{0,0,0}(\alpha_1,\alpha_2) = 
       q_1+n_1 n_2-2n_1-n_2.
\end{equation}
    Furthermore, $\Gamma_m(\alpha_1,\alpha_2)$ has parameters $\left (q_1 q_2, (q_1-1)(n_2+1), q_1-2+(n_1-1)n_2; n_1, q_1\right )$,
    and the number of common neighbours between a pair of nonadjacent vertices belongs to the set
    \[
    \left \{n_1(n_2+1)\right \} \cup \bigcup_{i=1}^{m-1} \left \{ 2n_1+ \mathcal X^{(m)}_{0,0,i}(\alpha_1,\alpha_2) \;\right \}.
    \]
\end{theorem}
\begin{proof}
    The graph $\Gamma = \Gamma_m(\alpha_1,\alpha_2)$ has connection set 
    $S = \mathsf C_1(\alpha_1) \cup \mathsf D_0(\alpha_1,\alpha_2)$.
    By Lemma~\ref{lem:undirect}, $\Gamma$ is undirected if and only if $q_1n_1 \equiv q_2n_2 \pmod 2$.
    Next, we show that $\Gamma$ has a regular clique.
    Define $\mathfrak C := \{(0,0)\} \cup \mathsf C_1(\alpha_1)$.
By \eqref{eqn:CD}, it follows that
\[
\underline{ \mathfrak C}\left (\underline{\mathsf C_1(\alpha_1) \cup \mathsf D_0(\alpha_1,\alpha_2)}\right ) = n_1\left ( \underline{\operatorname{GF}(q_1) \times \operatorname{GF}(q_2)} - \underline{\mathfrak C}\right )+(q_1-1)\underline{ \mathfrak C}.
\]
In other words, $\mathfrak C$ is an $n_1$-regular clique.

Let $(x,y) \in \operatorname{GF}(q_1) \times \operatorname{GF}(q_2)$.
By virtue of $\Gamma$ being a Cayley graph, the number of common neighbours of $(0,0)$ and $(x,y)$ is equal to the multiplicity of $(x,y)$ in 
$$\underline{S}^2 = \underline{\mathsf C_1(\alpha_1)}^2 + 2\underline{\mathsf C_1(\alpha_1)}\,\underline{\mathsf D_0(\alpha_1,\alpha_2)} + \underline{\mathsf D_0(\alpha_1,\alpha_2)}^2.$$
The rest of the statement of the corollary follows from Lemma~\ref{lem:Spart-rankm+3}.
\end{proof}

Theorem~\ref{thm:main} highlights $\mathcal X^{(m)}_{0,0,0}(\alpha_1,\alpha_2)$ as the key to determining whether or not the graph $\Gamma_m(\alpha_1,\alpha_2)$ is a Neumaier graph.
We can find explicit expressions for $\mathcal X^{(m)}_{0,0,0}(\alpha_1,\alpha_2)$ in the literature when $q_1$ and $q_2$ are powers of the same prime $p$.
Using these expressions, we recover previously discovered constructions of strongly regular graphs (see Remark~\ref{rem:srgOA} and Remark~\ref{rem:uniformsrg}, below).

\begin{theorem}[{\cite[Theorem 1]{vandiver}}]
\label{thm:vandiver}
    Let $m \geqslant 2$ and $n \geqslant 1$ be integers such that $q = 1+mn$ is a prime power.
    Let $\alpha$ be a primitive element of $\operatorname{GF}(q)$. 
    Then $\mathcal X^{(m)}_{0,0,0}(\alpha,\alpha) = (n-1)^2+n(m-1)$ and, for each $i \in \{1,\dots,m-1\}$, we have $\mathcal X^{(m)}_{0,0,i}(\alpha,\alpha) = n(n-1)$.
\end{theorem}

It follows immediately from Theorem~\ref{thm:main} and Theorem~\ref{thm:vandiver} that there exist strongly regular graphs of the form $\Gamma_m(\alpha_1,\alpha_2)$ when $\alpha_1 = \alpha_2$.

\begin{remark}
\label{rem:srgOA}
Let $m \geqslant 2$ be an integer.
    Suppose that $q = 1 + mn$ is a prime power.
    Let $\alpha$ be a primitive element for $\operatorname{GF}(q)$.
    Then $\Gamma_m(\alpha,\alpha)$ is a strongly regular graph with parameters
    \[
    \left (q^2, n(q+m-1), q-2+n(n-1), n(n+1)\right ).
    \]
    Thus, $\Gamma_m(\alpha,\alpha)$ is a strongly regular graph of Latin square type in the sense of \cite{vDM08}.
\end{remark}

Baumert et al.~\cite{BAUMERT1982} introduced the notion of \emph{uniform cyclotomy}, which we now describe.
The set of cyclotomic numbers $\left \{c_m(\alpha; i,j) \; : \; i,j \in \{1,\dots,m\}\right \}$ is called \textbf{uniform} if 
\begin{align*}
    c_m(\alpha; i,0) = c_m(\alpha; 0,i) &= c_m(\alpha; i,i) = c_m(\alpha; 0,1) \text{ for all $i \ne 0$; and} \\
    c_m(\alpha; i,j) &= c_m(\alpha; 1,2) \text{ for all $0 \ne i \ne j \ne 0$.}
\end{align*}
The following lemma is a special case of the results of Baumert et al.~\cite{BAUMERT1982}.
\begin{lemma}
\label{lem:quniform}
    Let $m \geqslant 2$ and $n \geqslant 1$ be integers such that $q = 1+mn$ is a prime power of a prime $p$.
    Let $\alpha$ be a primitive element of $\operatorname{GF}(q)$.
    Then $\left \{c_m(\alpha; i,j) \; : \; i,j \in \{1,\dots,m\}\right \}$ is uniform if and only if $-1$ is a power of $p$ modulo $m$.
    Furthermore, if $-1$ is a power of $p$ modulo $m$ then $qn$ is even, $q = r^2$ for some $r \equiv 1 \pmod m$, and
    \begin{align*}
        c_{m}(\alpha; 0,0) &= \frac{(r-1)^2}{m^2}-\frac{(m-3)(r-1)}{m}-1; \\
        c_{m}(\alpha; i,0) &= c_{m}(\alpha; 0,i) = c_{m}(\alpha; i,i) = \frac{(r-1)(r-1+m)}{m^2} \text{ for all $i \in \{1,\dots,m-1\}$}; \\
        c_{m}(\alpha; i,j) &= \frac{(r-1)^2}{m^2} \text{ for all $0 \ne i \ne j \ne 0$.}
    \end{align*}
\end{lemma}
We can use Lemma~\ref{lem:quniform} to show that the Neumaier graphs $\Gamma_m(\alpha_1,\alpha_2)$ where both sets of cyclotomic numbers are uniform must be strongly regular.

\begin{theorem}
\label{thm:uniformsrg}
Let $m \geqslant 2$ be an integer.
Let $q_1, q_2 \equiv 1\pmod m$ be prime powers
and $\alpha_1, \alpha_2$ be primitive elements of $\operatorname{GF}(q_1)$ and $\operatorname{GF}(q_2)$, respectively.
    Suppose that $\Gamma_{m}(\alpha_1,\alpha_2)$ is a Neumaier graph where both sets of cyclotomic numbers $\left \{c_m(\alpha_1; i,j) \; : \; i,j \in \{1,\dots,m\}\right \}$ and $\left \{c_m(\alpha_2; i,j) \; : \; i,j \in \{1,\dots,m\}\right \}$ are uniform.
    Then $\Gamma_{m}(\alpha_1,\alpha_2)$ is strongly regular.
\end{theorem}
\begin{proof}
     By Theorem~\ref{thm:main}, $q_1n_1 \equiv q_2n_2 \pmod 2$ and
$\mathcal X^{(m)}_{0,0,0}(\alpha_1,\alpha_2) = 
       q_1+n_1 n_2-2n_1-n_2$.
       Using Lemma~\ref{lem:quniform}, we can write $q_1 = r_1^2, q_2 = r_2^2$ for some $r_1,r_2 \equiv 1 \pmod m$.
       Furthermore, from Lemma~\ref{lem:quniform} it also follows that
\begin{align*}
    \mathcal X^{(m)}_{0,0,0}(\alpha_1,\alpha_2) - (q_1+n_1 n_2-2n_1-n_2)&= \frac{(r_1-r_2) (m-1) (r_2+(m-1)r_1)}{m^2}.
\end{align*}
Again, by Lemma~\ref{lem:quniform}, for each $i \in \{1,\dots,m-1\}$, we have
\[
\mathcal X^{(m)}_{0,0,i}(\alpha_1,\alpha_2) - n_1(n_2-1) = \frac{(r_1-r_2) (m-1) (r_2+(m-1)r_1)}{m^2}.
\]
Hence, by Theorem~\ref{thm:main}, the graph $\Gamma_m(\alpha_1,\alpha_2)$ is strongly regular.
\end{proof}

\begin{remark}
\label{rem:uniformsrg}
    One can extract from the proof of Theorem~\ref{thm:uniformsrg} that the resulting strongly regular graphs that are not covered by Remark~\ref{rem:srgOA} have (generalised Denniston) parameters $(v,k,\lambda,\mu)$, where
    \begin{align*}
        v&=q^{2(2l+1)}; \\
        k &= \frac{(q^{2l}-1)(q^{2(l+1)}-1)}{q+1}+q^{2l}-1; \\
        \lambda &= q^{2l}-2+\frac{q^{2(l+1)}-1}{q+1}\left (\frac{q^{2l}-1}{q+1}-1\right); \\
        \mu &= \frac{q^{2l}-1}{q+1}\left (\frac{q^{2(l+1)}-1}{q+1}+1\right),
    \end{align*}
    for some prime power $q$ and positive integer $l$.
    This family of strongly regular graphs contains as a subfamily the strongly regular graphs discovered earlier by Fern\' andez-Alcober et al.~\cite[Proposition 4.3]{FERNANDEZ2010} and Momihara~\cite[Proposition 2]{MOMIHARA}. 
    This family also forms part of a more general family described by Li et al.~\cite[Theorem 1.10]{li2025} in their recent preprint.
\end{remark}

All known instances of strongly regular $\Gamma_m(\alpha_1,\alpha_2)$ occur when the corresponding prime powers $q_1$ and $q_2$ are powers of the same prime.

\begin{question}
\label{q:coprime}
    Do there exist coprime prime powers $q_1 \equiv 1 \pmod m$ and $q_2 \equiv 1 \pmod m$ such that $\Gamma_m(\alpha_1,\alpha_2)$ is strongly regular for some choice of primitive elements $\alpha_1$ and $\alpha_2$ of $\operatorname{GF}(q_1)$ and $\operatorname{GF}(q_2)$, respectively? 
\end{question}

\section{Coherent rank 6 and beyond}
\label{sec:6beyond}

\subsection{More strongly regular graphs}

In view of Corollary~\ref{cor:coherentRanklem}, one might expect that for an appropriately chosen pair of prime powers and primitive elements $\alpha_1$ and $\alpha_2$, the graph $\Gamma_2(\alpha_1,\alpha_2)$ is a Neumaier graph with coherent rank $5$.
However, this is not the case, as we now briefly demonstrate.

\begin{theorem}[{\cite[Lemma 6]{Storer1967},\cite[Proposition 9]{Myerson81}}]
\label{thm:ord2cyc}
    Let $q$ be a prime power with $q=2n +1$ and $\alpha$ be a primitive element for $\operatorname{GF}(q)$. 
    Let $\mathfrak r_n$ be the remainder of $n$ after division by $2$. 
    Then
    \begin{align*}
        c_2(\alpha; 0,\mathfrak r_n) &= \frac{n-2+3\mathfrak r_n}{2}; \\
         c_2(\alpha; 0,1-\mathfrak r_n) = c_2(\alpha; 1,0)= c_2(\alpha; 1,1) &= \frac{n-\mathfrak r_n}{2}.
         \end{align*}
\end{theorem}

Now we are ready to dispose of the special case of Theorem~\ref{thm:main} when $m=2$.

\begin{corollary}
\label{cor:rank3}
Let $q_1$ and $q_2$ be prime powers, each congruent to $1 \pmod {2}$ and let $\alpha_1$ and $\alpha_2$ be primitive elements of $\operatorname{GF}(q_1)$ and $\operatorname{GF}(q_2)$, respectively.
    Then $\Gamma_2(\alpha_1,\alpha_2)$ is a Neumaier graph if and only if $ q_1 = q_2$.
\end{corollary}
\begin{proof}
    By Theorem~\ref{thm:ord2cyc}, we have $\mathcal X^{(2)}_{0,0,0}(\alpha_1,\alpha_2) = q_1+n_1n_2-2n_1-n_2$, which when combined with \eqref{eqn:edgeregular} yields $q_1 = q_2$.
    The converse follows from Remark~\ref{rem:srgOA}.
\end{proof}

The graphs resulting from Corollary~\ref{cor:rank3} are a subset of those of Remark~\ref{rem:srgOA}, which are all strongly regular.

\subsection{Coherent rank 6}
\label{sec:rank6C}

In this section, we present a new family of Neumaier graphs that have coherent rank 6 (see Corollary~\ref{cor:rank6}, below).
This family occurs as a special case of Theorem~\ref{thm:main} for $m=3$.

Let $q = p^r$ be a prime power with $q \equiv 1 \pmod 3$ and $\alpha$ be a primitive element for $\operatorname{GF}(q)$.
Define $u_3(q), v_3(q) \in \mathbb Z$ by $4q = u_3(q)^2+27v_3(q)^2$ and $u_3(q) \equiv 1 \pmod 3$, where $\operatorname{gcd}(u_3(q),p)  = 1$ if $p \equiv 1 \pmod 3$.
Note that $u_3(q)$ is uniquely determined while $v_3(q)$ is determined only up to sign~\cite[Proposition 10]{Myerson81}.
The sign for $v_3(q)$ can be chosen to agree with the choice of primitive element $\alpha$ in Theorem~\ref{thm:ord3cyc}, below, with respect to the value of cyclotomic numbers of order $3$.
We write $v_3(\alpha)$ for such a choice of $v_3(q)$.

\begin{theorem}[{\cite[Lemma 7]{Storer1967}, \cite[Proposition 10]{Myerson81}}]
\label{thm:ord3cyc}
    Let $q$ be a prime power with $q \equiv 1 \pmod 3$ and $\alpha$ be a primitive element for $\operatorname{GF}(q)$.
    \begin{align*}
        c_3(\alpha; 0,0) &= \frac{q-8+u_3(q)}{9}; \\
        c_3(\alpha; 0,1) = c_3(\alpha; 1,0) = c_3(\alpha; 2,2) &= \frac{2q-4-u_3(q)-9v_3(\alpha)}{18}; \\
        c_3(\alpha; 0,2) = c_3(\alpha; 2,0) = c_3(\alpha;1,1) &= \frac{2q-4-u_3(q)+9v_3(\alpha)}{18}; \\
        c_3(\alpha;1,2) = c_3(\alpha; 2,1) &= \frac{q+1+u_3(q)}{9}.
    \end{align*}
\end{theorem}
Define 
\begin{align*}
    \Phi_3(\alpha_1,\alpha_2) &:= \frac{u_3(q_1)u_3(q_2)+27v_3(\alpha_1)v_3(\alpha_2)}{4}; \\
    \Psi_3(\alpha_1,\alpha_2) &:= u_3(q_1)v_3(\alpha_2) - u_3(q_2)v_3(\alpha_1).
\end{align*}
Let $q_1 = 1+3n_1$ and $q_2 = 1+3n_2$ be prime powers.
Let $\alpha_1$ and $\alpha_2$ be primitive elements of $\operatorname{GF}(q_1)$ and $\operatorname{GF}(q_2)$, respectively.
By Theorem~\ref{thm:ord3cyc}, we have

\begin{align}
\label{eqn:particular3}
    \mathcal X^{(3)}_{0,0,0}(\alpha_1,\alpha_2) &= \frac{(q_1-2)(q_2-2)}{9}+\frac{2\Phi_3(\alpha_1,\alpha_2)}{9} + \frac{2}{3}; \\
    \label{eqn:particular31}
    \mathcal X^{(3)}_{0,0,1}(\alpha_1,\alpha_2) &= \frac{(q_1-2)(q_2-2)}{9}-\frac{\Phi_3(\alpha_1,\alpha_2)}{9} + \frac{\Psi_3(\alpha_1,\alpha_2)}{4}; \\
    \label{eqn:particular32}
    \mathcal X^{(3)}_{0,0,2}(\alpha_1,\alpha_2) &= \frac{(q_1-2)(q_2-2)}{9}-\frac{\Phi_3(\alpha_1,\alpha_2)}{9} -\frac{\Psi_3(\alpha_1,\alpha_2)}{4}.
\end{align}

Now we can prove the main result of this section, which provides a simple equation that can be used to check if $\Gamma_3(\alpha_1,\alpha_2)$ is a Neumaier graph.

\begin{corollary}
\label{cor:rank6}
Let $q_1$ and $q_2$ be prime powers, each congruent to $1 \pmod {3}$ and let $\alpha_1$ and $\alpha_2$ be primitive elements of $\operatorname{GF}(q_1)$ and $\operatorname{GF}(q_2)$, respectively.
    Then $\Gamma=\Gamma_3(\alpha_1,\alpha_2)$ is a Neumaier graph if and only if
        \begin{equation}
    \label{eqn:necCond}
        4(2q_1-q_2) = u_3(q_1)u_3(q_2)+ 27v_3(\alpha_1)v_3(\alpha_2).
    \end{equation}
    Furthermore, if $\operatorname{gcd}(q_1,q_2) = 1$ then $\Gamma$ has coherent rank 6.
\end{corollary}
\begin{proof}
    Equation~\eqref{eqn:necCond} follows from Theorem~\ref{thm:main} and \eqref{eqn:particular3}.
    By Corollary~\ref{cor:coherentRanklem}, the coherent rank of $\Gamma_m(\alpha_1,\alpha_2)$ is at most $6$.
Now suppose that $\operatorname{gcd}(q_1,q_2) = 1$.
Let $\mathcal S$ be the minimal Schur ring over $\operatorname{GF}(q_1) \times \operatorname{GF}(q_2)$ for which $\underline{\mathsf C_1(\alpha_1)} + \underline{\mathsf D_{0}(\alpha_1,\alpha_2 )} \in \mathcal S$.
It remains to show that $\operatorname{rk} (\mathcal S) \geqslant 6$.
     Let $S = \mathsf C_1(\alpha_1) \cup \mathsf D_{0}(\alpha_1,\alpha_2 )$.
    By Lemma~\ref{lem:Spart-rankm+3}, we have
    \begin{align}
    \label{eqn:S2}
        \underline{S}^2 = & \ k\underline{\{(0,0)\}} + \lambda_1 \underline{\mathsf C_1(\alpha_1)} + \lambda_2\underline{\mathsf D_0(\alpha_1,\alpha_2)} + n_1(n_2+1)\underline{\mathsf C_2(\alpha_2)} + \sum_{i=1}^{2} \mu_i \underline{\mathsf D_i(\alpha_1,\alpha_2)},
    \end{align}
    where $k = (q_1-1)(n_2+1)$, $\lambda_1 = q_1-2+(n_1-1)n_2$, $\lambda_2=2n_1 + \mathcal X_{0,0,0}^{(m)}(\alpha_1,\alpha_2)$, and $\mu_i = 2n_1 + \mathcal X_{0,0,i}^{(m)}(\alpha_1,\alpha_2)$ for each $i \in \{1,2\}$.
    Suppose (for a contradiction) that
\[
\mathcal X^{(3)}_{0,0,1}(\alpha_1,\alpha_2) = \mathcal X^{(3)}_{0,0,2}(\alpha_1,\alpha_2),
\]
    from which, together with \eqref{eqn:particular31} and \eqref{eqn:particular32}, it follows that $u_3(q_1)v_3(\alpha_2) = u_3(q_2)v_3(\alpha_1)$.

    In the case where $v_3(\alpha_2) = 0$, we must also have $v_3(\alpha_1) = 0$, since $u_3(q_2) \equiv 1 \pmod 3$.
    It follows that both $q_1$ and $q_2$ are powers of $2$.
    Otherwise, suppose $v_3(\alpha_2) \ne 0$.
    Then $u_3(q_1) = u_3(q_2)v_3(\alpha_1)/v_3(\alpha_2)$.
    It follows that $q_1v_3(\alpha_2)^2 = v_3(\alpha_1)^2q_2$.
    Now if $q_1$ does not divide $q_2$ then $q_1$ must divide $v_3(\alpha_1)^2$, i.e., $v_3(\alpha_1)^2 = sq_1$ for some positive integer $s$.
    But then $4q_1 = u_3(q_1)^2 + 27sq_1$, which contradicts the fact that $u_3(q_1)^2 > 0$.
    Thus, $\mathcal X^{(3)}_{0,0,1}(\alpha_1,\alpha_2) \ne \mathcal X^{(3)}_{0,0,2}(\alpha_1,\alpha_2)$.
    Applying Lemma~\ref{lem:Wielandt} to \eqref{eqn:S2}, we can deduce that $\underline{\mathsf D_i(\alpha_1,\alpha_2)} \in \mathcal S$ for some $i \in \{1,2\}$.
    To ease the notation of this proof, we shall reduce subscripts modulo 3.

    By Lemma~\ref{lem:Spart-rankm+3}, we have
        \begin{align}
        \label{eqn:di2}
            \begin{split}
                \underline{\mathsf D_i(\alpha_1,\alpha_2)}^2 =& \ 3n_1n_2 \underline{\{(0,0\}} + (n_1-1)n_2\underline{\mathsf C_1(\alpha_1)}+n_1(n_2-1)\underline{\mathsf C_2(\alpha_2)}\\
                & + \sum_{j=0}^{2}\mathcal X^{(3)}_{i,i,j}(\alpha_1,\alpha_2)\underline{\mathsf D_j(\alpha_1,\alpha_2)} 
            \end{split} \\
            \begin{split}
            \label{eqn:diS}
                \underline{\mathsf D_i(\alpha_1,\alpha_2)}\underline{S} =& 
                \left ( n_1+\mathcal X^{(3)}_{0,i,0}(\alpha_1,\alpha_2) \right )\underline{\mathsf D_0(\alpha_1,\alpha_2)} + \left ( n_1+\mathcal X^{(3)}_{0,i,-i}(\alpha_1,\alpha_2) \right )\underline{\mathsf D_{-i}(\alpha_1,\alpha_2)} \\
                &+ n_1n_2 \underline{C_1} + n_1(n_2+1) \underline{C_2} + \left (n_1-1+\mathcal X^{(3)}_{0,i,i}(\alpha_1,\alpha_2) \right )\underline{\mathsf D_i(\alpha_1,\alpha_2)}.
            \end{split}
        \end{align}
        Next, we claim that $\underline{\mathsf C_1(\alpha_1)}$ and $\underline{\mathsf D_0(\alpha_1,\alpha_2)}$ must have distinct coefficients in at least one of \eqref{eqn:di2} and \eqref{eqn:diS}.
        Indeed, suppose to the contrary that $(n_1-1)n_2 = \mathcal X^{(3)}_{i,i,0}(\alpha_1,\alpha_2)$ and $n_1n_2 = n_1 + \mathcal X^{(3)}_{0,i,0}(\alpha_1,\alpha_2)$.
        Using Proposition~\ref{pro:Xeq} and Theorem~\ref{thm:cycBasic} (i) these equations become $(n_1-1)n_2 = \mathcal X^{(3)}_{0,0,-i}(\alpha_1,\alpha_2)$ and $n_1(n_2-1) = \mathcal X^{(3)}_{0,0,i}(\alpha_1,\alpha_2)$.
        Using \eqref{eqn:particular31}, \eqref{eqn:particular32}, and \eqref{eqn:necCond}, it follows that $q_1 = q_2$, a contradiction.
        Now we can apply Lemma~\ref{lem:Wielandt} to deduce that both $\underline{\mathsf C_1(\alpha_1)}$ and $\underline{\mathsf D_0(\alpha_1,\alpha_2)}$ belong to $\mathcal S$.
        Thus, 
        \begin{align}
        \label{eqn:d02}
            \begin{split}
                \underline{\mathsf D_0(\alpha_1,\alpha_2)}^2 =& \ 3n_1n_2 \underline{\{(0,0\}} + (n_1-1)n_2\underline{\mathsf C_1(\alpha_1)}+n_1(n_2-1)\underline{\mathsf C_2(\alpha_2)}\\
                & + \sum_{j=0}^{2}\mathcal X^{(3)}_{0,0,j}(\alpha_1,\alpha_2)\underline{\mathsf D_j(\alpha_1,\alpha_2)} 
            \end{split}
        \end{align}
        is in $\mathcal S$.
        Finally, we claim that $\underline{\mathsf C_2(\alpha_1)}$ and $\underline{\mathsf D_{-i}(\alpha_1,\alpha_2)}$ must have distinct coefficients in at least one of \eqref{eqn:di2} and \eqref{eqn:d02}.
        Indeed, suppose to the contrary that $n_1(n_2-1) = \mathcal X^{(3)}_{i,i,-i}(\alpha_1,\alpha_2)$ and $n_1(n_2-1) = \mathcal X^{(3)}_{0,0,-i}(\alpha_1,\alpha_2)$, a contradiction.
        Apply Lemma~\ref{lem:Wielandt} to deduce that both $\underline{\mathsf C_2(\alpha_1)}$ and $\underline{\mathsf D_{-i}(\alpha_1,\alpha_2)}$ belong to $\mathcal S$.
\end{proof}

Corollary~\ref{cor:rank6} shows us that $\Gamma_3(\alpha_1,\alpha_2)$ has support of cardinality 2.
Furthermore, one can easily extend the proof of Corollary~\ref{cor:rank6} to verify that the case when $\Gamma_3(\alpha_1,\alpha_2)$ is strongly regular is covered by Remarks~\ref{rem:srgOA} and \ref{rem:uniformsrg}.
\begin{question}
    \label{prob:inf6}
    Do there exist infinitely many pairs of coprime prime powers $q_1, q_2 \equiv 1 \pmod 3$ for which \eqref{eqn:necCond} is satisfied for some primitive elements $\alpha_1,\alpha_2$ of $\operatorname{GF}(q_1)$ and $\operatorname{GF}(q_2)$, respectively?
\end{question}

Based on our computational investigation below, we conjecture that the answer to Question~\ref{prob:inf6} is yes.
\begin{lemma}
\label{lem:condBound}
Let $q_1$ and $q_2$ be prime powers, each congruent to $1 \pmod {3}$ and let $\alpha_1$ and $\alpha_2$ be primitive elements of $\operatorname{GF}(q_1)$ and $\operatorname{GF}(q_2)$, respectively.
Suppose that \eqref{eqn:necCond} is satisfied.
Then $(u_3(q_2)+u_3(q_1)/2)^2 \leqslant 9q_1$ and $(v_3(\alpha_2)+v_3(\alpha_1)/2)^2 \leqslant q_1/3$.
\end{lemma}
\begin{proof}
    Equation~\eqref{eqn:necCond} is equivalent to the equation
\begin{equation}
    \label{eqn:cond}
    -27(v_3(\alpha_2)+2v_3(q_1))(v_3(\alpha_2)-v_3(q_1)) = (u_3(\alpha_2)+2u_3(q_1))(u_3(\alpha_2)-u_3(q_1)).
\end{equation}
    The lemma follows from the fact that the left-hand side of \eqref{eqn:cond} must be at least the minimum value of the right-hand side of \eqref{eqn:cond} and the right-hand side of \eqref{eqn:cond} must be at most the maximum value of the left-hand side of \eqref{eqn:cond}.
\end{proof}

Lemma~\ref{lem:condBound} implies that for any fixed prime power $q_1 \equiv 1 \pmod 3$, there are only finitely many prime powers $q_2 \equiv 1 \pmod 3$ for which \eqref{eqn:necCond} can be satisfied.
This means that given a prime power $q_1 \equiv 1 \pmod 3$, one can efficiently produce a list of prime powers $q_2 \equiv 1 \pmod 3$ that satisfy \eqref{eqn:necCond} together with $q_1$.
See Table~\ref{tab:list} for the result of such a computation for $q_1$, a prime power less than 250.

\begin{table}[htbp]
    \centering
    \begin{tabular}{c|c || c | c || c | c || c | c}
        $q_1$ & $q_2$ &  $q_1$ & $q_2$  &  $q_1$ & $q_2$ &  $q_1$ & $q_2$  \\
        \hline
        4 & 7  &   13 & 49 &    61 & 97   &   163 & 211 \\
        4 & 13 &   19 & 67 &    97 & 241  &   169 & 313 \\ 
        7 & 16 &   31 & 43 &    109 & 443 &   193 & 769 \\
        7 & 19 &   37 & 73 &    139 & 331 &   199 & 787\\
        13 & 16 &  49 & 193 &    151 & 163&    223 & 811 
    \end{tabular}
    \caption{Pairs of prime powers $q_1, q_2$ for which \eqref{eqn:necCond} is satisfied for some primitive elements $\alpha_1,\alpha_2$ of $\operatorname{GF}(q_1)$ and $\operatorname{GF}(q_2)$, respectively.}
    \label{tab:list}
\end{table}

Each pair of prime powers in Table~\ref{tab:list} corresponds to a Neumaier graph with coherent rank 6 via Corollary~\ref{cor:rank6}.

One can perform further analysis on particular prime powers, such as the following lemma, which classifies even solutions to \eqref{eqn:necCond}.

\begin{lemma}
    Suppose that $q_1$ and $q_2$ are coprime prime powers each congruent to $1 \pmod 3$ that satisfy \eqref{eqn:necCond}.
    Suppose that $q_1q_2$ is even.
    Then $(q_1,q_2)$ is equal to $(4,7)$, $(4,13)$, $(7,16)$, or $(13,16)$.
\end{lemma}
\begin{proof}
    Suppose that $q_1 \equiv 1 \pmod 3$ is an even prime power then $q_1 = 2^{2r}$ for some $r \in \mathbb N$.
Hence, $u_3(q_1) = (-2)^{r+1}$ and $v_3(q_1) = 0$.
Then \eqref{eqn:necCond} becomes 
\[
2^{2r+1}-q_2 = (-2)^{r-1}u_3(q_2).
\]
If $r > 1$ then $q_2$ must be even.
Otherwise, if $r = 1$ then we have $q_1 = 4$ and $2^{3}-q_2 = u_3(q_2)$.
It follows that $-27v_3^2(q_2) = (u_3(q_2) +8)(u_3(q_2) -4)$.
Hence, $q_2 = 7$ or $q_2 = 13$.

The proof for the case when $q_2$ is even follows in a similar fashion.
\end{proof}

\subsection{Coherent rank 7}

In this section, we consider Neumaier graphs that have coherent rank 7.
Before we present the result of our new construction, we give examples of Neumaier graphs in the literature whose coherent rank is $7$. 

\begin{example}
Consider the direct product of the additive groups $\mathbb Z/2\mathbb Z$ and $\mathbb Z/8\mathbb Z$.
Define the subsets $\mathsf S_1 = \{(0,0)\}$, $\mathsf S_2 = \{(1,0)\}$, $\mathsf S_3 = \{(0,4)\}$, $\mathsf S_4 = \{(1,4)\}$,
\begin{align*}
 \mathsf S_5 &= \{(0,1),(0,7),(1,1),(1,7)\}, \\
    \mathsf S_6 &= \{(0,2),(0,6),(1,2),(1,6)\}, \\
    \mathsf S_7 &= \{(0,3),(0,5),(1,3),(1,5)\}.
\end{align*}

The graph $\Omega = \operatorname{Cay}(\mathbb Z/2\mathbb Z \times \mathbb Z/8\mathbb Z,\mathsf S_4 \cup \mathsf S_5 \cup \mathsf S_6)$ is a Neumaier graph with parameters $(16,9,2; 2,4)$.
Evans et al.~\cite{EGP19} showed that $\Omega$ is the only Neumaier graph having parameters $(16,9,2; 2,4)$ and that there is no strictly Neumaier graph on less than 16 vertices.
The minimal Schur ring over $\mathbb Z/2\mathbb Z \times \mathbb Z/8\mathbb Z$ that contains $\mathsf S_4 \cup \mathsf S_5 \cup \mathsf S_6$ has basic sets $\mathsf S_1, \dots, \mathsf S_7$.
Hence, $\Omega$ has coherent rank 7.
The graph $\Omega$ can also be expressed as a Cayley graph over the dihedral group $D_{16}$~\cite{jazaeri2023neumaiercayleygraphs}.
\end{example}

Abiad et al.~\cite{abiad2023infinite} produced a family of Neumaier graphs that contains two examples that have coherent rank $7$.
Let $p$ and $q$ be distinct primes, let $\alpha$ be a primitive root of both $p$ and $q$, and let $(p-1)(q-1) = mn$ with $m = \operatorname{gcd}(p-1,q-1)$.
By \cite[Lemma 1]{Whiteman}, there exists an integer $x$ such that the sets $\mathsf K_{i}(\alpha) := \{x^i\alpha^j \pmod {pq} \;:\; 0 \leqslant j \leqslant  n-1\}$ with $i \in \{0,\dots,m-1\}$ partition $(\mathbb Z/pq\mathbb Z)^*$, where $R^*$ denotes the subset of invertible elements of a ring $R$.

A $\textbf{spread}$ of cocliques of a graph $\Gamma$ is a partition of the vertex set of $\Gamma$, each of whose parts induces a coclique.
Let $\Gamma$ be a graph having a spread of cocliques $C_1,\dots,C_p$.
Denote $t$ copies of $\Gamma$ by $\Gamma_1,\dots,\Gamma_t$ with $C_{i,1},\dots,C_{i,p}$ denoting the spread of cocliques in the $i$th copy $\Gamma_i$.
For permutations $\pi_2,\dots,\pi_t \in \operatorname{Sym}(p)$, define $F_{(\pi_2,\dots,\pi_t)}(\Gamma)$ to be the disjoint union of the graphs $\Gamma_1,\dots,\Gamma_t$ and for each $k \in \{1,\dots,p\}$ add the edges between all vertices in $C_{1,k} \cup C_{2,\pi_2(k)}\cup \dots \cup C_{t,\pi_t(k)}$.

By \cite[Theorem 4.7]{abiad2023infinite}, the graph $\operatorname{Cay}(\mathbb Z/pq\mathbb Z,\mathsf K_{0}(\alpha))$ has a spread of cocliques each of order $q$.
Furthermore, one can obtain Neumaier graphs using the following theorem.

\begin{theorem}[{cf.~\cite[Theorem 4.9]{abiad2023infinite}}]
\label{thm:adbkwz}
    Let $p \ne q$ be odd primes such that $q-1$ divides $p-1$ and let $\alpha$ be a primitive root of both $p$ and $q$.
    Suppose $t = (|\mathsf K_{0}(\alpha) \cap (\mathsf K_{0}(\alpha)+1)|+2)/q$ is an integer.
    Then, for $\pi_i \in \operatorname{Sym}(p)$ the graph $F_{(\pi_2,\dots,\pi_t)}(\operatorname{Cay}(\mathbb Z/pq\mathbb Z,\mathsf K_{0}(\alpha)))$ is a Neumaier graph with parameters $(tpq,p+tq-2,tq-2; 1, tq)$.
\end{theorem}

Among the infinite family of Neumaier graphs presented in \cite{abiad2023infinite}, one can find examples that have coherent rank $7$.

\begin{example}
    By Theorem~\ref{thm:adbkwz}, the graph $F_{()}(\operatorname{Cay}(\mathbb Z/65\mathbb Z,\mathsf K_{0}(2)))$ is a Neumaier graph with parameters $(65,16,3; 1, 5)$ and $F_{()}(\operatorname{Cay}(\mathbb Z/185\mathbb Z,\mathsf K_{0}(2)))$ is a Neumaier graph with parameters $(185,40,3; 1, 5)$.
    These graphs can be expressed as $\operatorname{Cay}(\mathbb Z/65\mathbb Z,\mathsf K_{0}(2) \cup 13(\mathbb Z/65\mathbb Z)^*)$ and $\operatorname{Cay}(\mathbb Z/185\mathbb Z,\mathsf K_{0}(2) \cup 37(\mathbb Z/185\mathbb Z)^*)$, respectively.
    We claim that both of these graphs have coherent rank 7.
    Indeed, one can check that the sets 
    \[
    \{0\}, p(\mathbb Z/pq\mathbb Z)^*, q(\mathbb Z/pq\mathbb Z)^*, \mathsf K_{0}(\alpha), \dots, \mathsf K_{m-1}(\alpha)
    \]
    form a Schur partition of $\mathbb Z/pq\mathbb Z$.
    The corresponding Schur ring has rank $q+2$.
    Moreover, using expressions generalised cyclotomic numbers from \cite{Whiteman} (or simply running the Weisfeiler-Leman algorithm~\cite{weisfeiler1968reduction}), one can show that the graphs $\operatorname{Cay}(\mathbb Z/65\mathbb Z,\mathsf K_{0}(2) \cup 13(\mathbb Z/65\mathbb Z)^*)$ and $\operatorname{Cay}(\mathbb Z/185\mathbb Z,\mathsf K_{0}(2) \cup 37(\mathbb Z/185\mathbb Z)^*)$ have coherent rank at least $7$.
\end{example}
Let $q = p^r$ be a prime power with $q \equiv 1 \pmod 4$ and let $\alpha$ be a primitive element for $\operatorname{GF}(q)$.
Define $u_4(q), v_4(q) \in \mathbb Z$ by $q = u_4(q)^2+4v_4(q)^2$ and $u_4(q) \equiv 1 \pmod 4$, where $\operatorname{gcd}(u_4(q),p)  = 1$ if $p \equiv 1 \pmod 4$.
Note that $u_4(q)$ is uniquely determined while $v_4(q)$ is determined only up to sign~\cite[Proposition 11]{Myerson81}.
The sign for $v_4(q)$ can be chosen to agree with the choice of primitive element $\alpha$ in Theorem~\ref{thm:ord4cyc}, below, with respect to the value of cyclotomic numbers of order $4$.
We write $v_4(\alpha)$ for such a choice of $v_4(q)$.

\begin{theorem}[{\cite[Lemma 19 and Lemma 19']{Storer1967},\cite[Proposition 11]{Myerson81}}]
\label{thm:ord4cyc}
    Let $q$ be a prime power with $q=4n +1$ and $\alpha$ be a primitive element for $\operatorname{GF}(q)$. 
    Let $\mathfrak r_n$ be the remainder of $n$ after division by $2$. 
    Then
    \begin{align*}
        c_4(\alpha; 0,2\mathfrak r_n) &= \frac{q-6u_4(q)-11+12\mathfrak r_n}{16}; \\
         c_4(\alpha; 1,1+2\mathfrak r_n) = c_4(\alpha; 0,3+2\mathfrak r_n)= c_4(\alpha; 3,2\mathfrak r_n) &= \frac{q+2u_4(q)-3-8v_4(\alpha)+4\mathfrak r_n}{16}; \\
         c_4(\alpha;2,2+2\mathfrak r_n) = c_4(\alpha; 0,2+2\mathfrak r_n) = c_4(\alpha; 2,2\mathfrak r_n) &= \frac{q+2u_4(q)-3-4\mathfrak r_n}{16}; \\
       c_4(\alpha; 3,3+2\mathfrak r_n) = c_4(\alpha; 0,1+2\mathfrak r_n) = c_4(\alpha; 1,2\mathfrak r_n) &= \frac{q+2u_4(q)-3+8v_4(\alpha)+4\mathfrak r_n}{16};
    \end{align*}
    and for each $(i,j) \in \{1,2,3\}^2$ with $i \ne j$, we have
    \[
            c_4(\alpha; i,j+2\mathfrak r_n) = \frac{q+2u_4(q)+1-4\mathfrak r_n(u_4(q)+1)}{16}.
    \]
\end{theorem}

Define 
\begin{align*}
    \Phi_4(\alpha_1,\alpha_2) &:= u_4(q_1)u_4(q_2)+4v_4(\alpha_1)v_4(\alpha_2); \\
    \Psi_4(\alpha_1,\alpha_2) &:= u_4(q_1)v_4(\alpha_2) - u_4(q_2)v_4(\alpha_1).
\end{align*}
Let $q_1 = 1+4n_1$ and $q_2 = 1+4n_2$ be prime powers.
Let $\alpha_1$ and $\alpha_2$ be primitive elements of $\operatorname{GF}(q_1)$ and $\operatorname{GF}(q_2)$, respectively.
Suppose that $n_1 \equiv n_2 \pmod 2$.
Then, by Theorem~\ref{thm:ord4cyc}, we have
\begin{align*}
     \mathcal X^{(4)}_{0,0,0}(\alpha_1,\alpha_2) & =\frac{(q_1-2)(q_2-2)}{16}+\frac{3\Phi_4(\alpha_1,\alpha_2)}{8} + \frac{9}{16}; \\ 
    \mathcal X^{(4)}_{0,0,1}(\alpha_1,\alpha_2) & =  \frac{(q_1-2)(q_2-2)}{16}-\frac{\Phi_4(\alpha_1,\alpha_2)}{8}+ \frac{1}{16}+(-1)^{n_1}\frac{\Psi_4(\alpha_1,\alpha_2)}{2};\\
    \mathcal X^{(4)}_{0,0,2}(\alpha_1,\alpha_2) & =  \frac{(q_1-2)(q_2-2)}{16}-\frac{\Phi_4(\alpha_1,\alpha_2)}{8}+ \frac{1}{16}; \\ 
        \mathcal X^{(4)}_{0,0,3}(\alpha_1,\alpha_2) & =  \frac{(q_1-2)(q_2-2)}{16}-\frac{\Phi_4(\alpha_1,\alpha_2)}{8}+ \frac{1}{16} - (-1)^{n_1}\frac{\Psi_4(\alpha_1,\alpha_2)}{2}.
\end{align*}

In order to guarantee coherent rank 7, in Corollary~\ref{cor:rank7}, below, we will require the following lemma.

\begin{lemma}
\label{lem:fullRank}
    Let $m \geqslant 4$ and $n_1,n_2 \geqslant 1$ be integers.
Suppose that $q_1 = 1+mn_1$ and $q_2 = 1+mn_2$ be prime powers.
Let $\alpha_1$ and $\alpha_2$ be primitive elements of $\operatorname{GF}(q_1)$ and $\operatorname{GF}(q_2)$, respectively.
Suppose that $\mathcal X^{(m)}_{0,0,i}(\alpha_1,\alpha_2)$ are pairwise distinct for each $i \in \{1,\dots,m-1\}$.
Let $\mathcal S$ be the minimal Schur ring over $\operatorname{GF}(q_1) \times \operatorname{GF}(q_2)$ for which $\underline{\mathsf C_1(\alpha_1)} + \underline{\mathsf D_{0}(\alpha_1,\alpha_2 )} \in \mathcal S$.
Then the basic sets of $\mathcal S$ are 
\[
\{(0,0)\},\mathsf C_1(\alpha_1), \mathsf C_2(\alpha_2), \mathsf D_{0}(\alpha_1,\alpha_2 ), \dots, \mathsf D_{m-1}(\alpha_1,\alpha_2 ).
\]
\end{lemma}
\begin{proof}
    We give the proof assuming that $q_1n_1 \equiv q_2n_2 \pmod 2$.
    The proof for the case $q_1n_1 \not \equiv q_2n_2 \pmod 2$ is similar, and we leave this to the reader.

    By Corollary~\ref{cor:coherentRanklem}, the rank of $\mathcal S$ is at most $m+3$.
    Let $S = \mathsf C_1(\alpha_1) \cup \mathsf D_{0}(\alpha_1,\alpha_2 )$.
    By Lemma~\ref{lem:Spart-rankm+3}, we have
    \begin{align*}
        \underline{S}^2 = & \ k\underline{\{(0,0)\}} + \lambda_1 \underline{\mathsf C_1(\alpha_1)} + \lambda_2\underline{\mathsf D_0(\alpha_1,\alpha_2)} + n_1(n_2+1)\underline{\mathsf C_2(\alpha_2)} + \sum_{i=1}^{m-1} \mu_i \underline{\mathsf D_i(\alpha_1,\alpha_2)},
    \end{align*}
    where $k = (q_1-1)(n_2+1)$, $\lambda_1 = q_1-2+(n_1-1)n_2$, $\lambda_2=2n_1 + \mathcal X_{0,0,0}^{(m)}(\alpha_1,\alpha_2)$, and $\mu_i = 2n_1 + \mathcal X_{0,0,i}^{(m)}(\alpha_1,\alpha_2)$ for each $i \in \{1,\dots,m-1\}$.
    Since $m \geqslant 4$, by Lemma~\ref{lem:Wielandt}, there exist at least two distinct indices $j$ and $k$ such that $\underline{\mathsf D_j(\alpha_1,\alpha_2)} \in \mathcal S$ and $\underline{\mathsf D_k(\alpha_1,\alpha_2)} \in \mathcal S$.
    By Lemma~\ref{lem:Spart-rankm+3}, we have
        \begin{equation}
    \label{eqn:dj2}
        \begin{split}
            \underline{\mathsf D_j(\alpha_1,\alpha_2)}^2 =& \ mn_1n_2\underline{\{(0,0)\}} +  (n_1-1)n_2\underline{\mathsf C_1(\alpha_1)}+n_1(n_2-1)\underline{\mathsf C_2(\alpha_2)} \\
            & + \sum_{i=0}^{m-1}\mathcal X^{(m)}_{j,j,i}(\alpha_1,\alpha_2)\underline{\mathsf D_i(\alpha_1,\alpha_2)} 
        \end{split}
    \end{equation}
            \begin{equation}
    \label{eqn:dk2}
        \begin{split}
            \underline{\mathsf D_k(\alpha_1,\alpha_2)}^2 =& \ mn_1n_2\underline{\{(0,0)\}} + (n_1-1)n_2\underline{\mathsf C_1(\alpha_1)}+n_1(n_2-1)\underline{\mathsf C_2(\alpha_2)} \\
            & + \sum_{i=0}^{m-1}\mathcal X^{(m)}_{k,k,i}(\alpha_1,\alpha_2)\underline{\mathsf D_i(\alpha_1,\alpha_2)} 
        \end{split}
    \end{equation}
    Since, by Proposition~\ref{pro:Xeq}, we have $\mathcal X^{(m)}_{j,j,0}(\alpha_1,\alpha_2) \ne \mathcal X^{(m)}_{k,k,0}(\alpha_1,\alpha_2)$, the coefficients of $\underline{\mathsf C_1(\alpha_1)}$ and $\underline{\mathsf D_0(\alpha_1,\alpha_2)}$ must be distinct in at least one of \eqref{eqn:dj2} and \eqref{eqn:dk2}.
    Now we can deduce, using Lemma~\ref{lem:Wielandt} that both $\underline{\mathsf C_1(\alpha_1)}$ and $\underline{\mathsf D_0(\alpha_1,\alpha_2)}$ belong to $\mathcal S$. 
    A similar argument applies to the coefficients of $\underline{\mathsf C_2(\alpha_1)}$ and $\underline{\mathsf D_i(\alpha_1,\alpha_2)}$ for each $i \in \{1,\dots,m-1\}$ and the conclusion can be obtained by applying Lemma~\ref{lem:Wielandt}.
\end{proof}

Now we can prove the counterpart to Corollary~\ref{cor:rank6}, above.

\begin{corollary}
\label{cor:rank7}
Let $q_1 = 1+4n_1$ and $q_2 = 1+4n_2$ be prime powers.
Let $\alpha_1$ and $\alpha_2$ be primitive elements of $\operatorname{GF}(q_1)$ and $\operatorname{GF}(q_2)$, respectively.
    Then $\Gamma=\Gamma_4(\alpha_1,\alpha_2)$ is a Neumaier graph if and only if $n_1 \equiv n_2 \pmod 2$ and
        \begin{equation}        
    \label{eqn:necCondRank7}
\frac{3q_1-q_2}{2} = u_4(q_1)u_4(q_2)+ 4v_4(\alpha_1)v_4(\alpha_2).     
\end{equation}
    Furthermore, if $\operatorname{gcd}(q_1,q_2) = 1$ then $\Gamma$ has coherent rank 7.
\end{corollary}
\begin{proof}
Equation~\eqref{eqn:necCondRank7} follows from Theorem~\ref{thm:main} and \eqref{eqn:particular3}.
    By Corollary~\ref{cor:coherentRanklem}, the coherent rank of $\Gamma_m(\alpha_1,\alpha_2)$ is at most $7$.
    It therefore suffices to show that the coherent rank of $\Gamma$ is at least $7$.
    
Suppose that $\mathcal X^{(4)}_{0,0,1}(\alpha_1,\alpha_2) = \mathcal X^{(4)}_{0,0,3}(\alpha_1,\alpha_2)$ then $u_4(q_1)v_4(\alpha_2) = u_4(q_2)v_4(\alpha_1)$.
In the case where $v_4(\alpha_2) = 0$, we must also have $v_4(\alpha_1) = 0$, since $u_4(q_2) \equiv 1 \pmod 4$.
    Using \eqref{eqn:necCondRank7}, it follows that $q_1$ is a multiple of $q_2$, a contradiction.
    Otherwise, suppose $v_4(\alpha_2) \ne 0$.
    Then $u_4(q_1) = u_4(q_2)v_4(\alpha_1)/v_4(\alpha_2)$.
    It follows that $q_1v_4(\alpha_2)^2 = v_4(\alpha_1)^2q_2$.
    Now if $q_1$ does not divide $q_2$ then $q_1$ must divide $v_4(\alpha_1)^2$, i.e., $v_4(\alpha_1)^2 = sq_1$ for some positive integer $s$.
    But then $4q_1 = u_4(q_1)^2 + 4sq_1$, which contradicts the fact that $u_4(q_1)^2 > 0$.

    Therefore, we may assume that $\mathcal X^{(4)}_{0,0,1}(\alpha_1,\alpha_2)$, $\mathcal X^{(4)}_{0,0,2}(\alpha_1,\alpha_2)$, and $\mathcal X^{(4)}_{0,0,3}(\alpha_1,\alpha_2)$ are pairwise distinct.
    The lemma then follows from Lemma~\ref{lem:fullRank}.
\end{proof}

\begin{question}
    \label{prob:inf7}
    Do there exist infinitely many pairs of coprime prime powers $q_1 = 4n_1+1$ and $q_2 = 4n_2+1$ with $n_1 \equiv n_2 \pmod 2$ for which \eqref{eqn:necCondRank7} is satisfied for some primitive elements $\alpha_1$ and $\alpha_2$ of $\operatorname{GF}(q_1)$ and $\operatorname{GF}(q_2)$, respectively?
\end{question}

One can, analogously to Lemma~\ref{lem:condBound}, show that, for each prime power $q_1$, there are finitely many prime powers $q_2$ for which \eqref{eqn:necCondRank7} is satisfied.
Indeed, the proof of the following lemma can be obtained \textit{mutatis mutandis} from the proof of Lemma~~\ref{lem:condBound}.

\begin{lemma}
\label{lem:condBound4}
Let $q_1$ and $q_2$ be prime powers, each congruent to $1 \pmod {4}$ and let $\alpha_1$ and $\alpha_2$ be primitive elements of $\operatorname{GF}(q_1)$ and $\operatorname{GF}(q_2)$, respectively.
Suppose that \eqref{eqn:necCondRank7} is satisfied.
Then $(u_4(q_2)+u_4(q_1))^2 \leqslant 4q_1$ and $(v_4(\alpha_2)+v_4(\alpha_1))^2 \leqslant q_1$.
\end{lemma}

In Table~\ref{tab:list7}, we list pairs of prime powers $q_1 = 4n_1+1$ and $q_2 = 4n_2+1$ with $n_1 \equiv n_2 \pmod 2$ that satisfy \eqref{eqn:necCondRank7}, where $q_1$ is at most $250$.

\begin{table}[htbp]
    \centering
    \begin{tabular}{c|c || c | c || c | c || c | c}
        $q_1$ & $q_2$ &  $q_1$ & $q_2$  &  $q_1$ & $q_2$ &  $q_1$ & $q_2$  \\
        \hline
         5 & 13 &   29& 229  &     89& 601 &   125 & 1093 \\
         5 & 37 &   41& 169  &    101& 109 &   149& 541 \\ 
        17 & 25 &   41& 241  &    109& 181 &   169& 1321 \\
        25 & 97 &   61& 349  &    113& 625 &   181& 829 \\
        29& 61  &   89& 289  &    125& 157 &   229& 2029 
    \end{tabular}
    \caption{Pairs of prime powers $q_1 = 4n_1+1$ and $q_2 = 4n_2+1$ with $n_1 \equiv n_2 \pmod 2$ for which \eqref{eqn:necCondRank7} is satisfied for some primitive elements $\alpha_1$ and $\alpha_2$ of $\operatorname{GF}(q_1)$ and $\operatorname{GF}(q_2)$, respectively.}
    \label{tab:list7}
\end{table}
\subsection{Beyond coherent rank 7}

In this section, we conclude with three lines of inquiry that have emerged naturally from our investigation.

\paragraph{Sums of products of cyclotomic numbers.}

For small values of $m$, closed-form expressions exist for the corresponding cyclotomic numbers.
For example, when $m=3$ or $m = 4$, we have Theorem~\ref{thm:ord3cyc} and Theorem~\ref{thm:ord4cyc}.
In these cases, we can obtain closed-form expressions for $\mathcal X_{i,j,k}^{(m)}(\alpha_1,\alpha_2)$ and establish stronger forms of Theorem~\ref{thm:main}.

\begin{problem}
     \label{prob:X}
    Find an explicit general closed-form expression for $\mathcal X_{i,j,k}^{(m)}(\alpha_1,\alpha_2)$, where $\alpha_1$ and $\alpha_2$ are primitive elements for $\operatorname{GF}(q_1)$ and $\operatorname{GF}(q_2)$, respectively.
\end{problem}

Denote by $\varphi(\cdot)$ Euler's totient function.
Let $q_1 = 1+mn_1$ and $q_2 = 1+mn_2$ be prime powers.
Let $\alpha_1$ and $\alpha_2$ be primitive elements of $\operatorname{GF}(q_1)$ and $\operatorname{GF}(q_2)$ respectively.
Based on the expressions we have for $m=3$, $4$, $5$, and $6$, using \cite[Theorem 11]{cyc5} and \cite[Theorem 15]{Storer1967} we conjecture that
\[
\mathcal X_{0,0,0}^{(m)}(\alpha_1,\alpha_2) = \frac{(q_1-2)(q_2-2) +(m-1)\left (3+(m-2)\mathbf v(\alpha_1)^\transpose D_m \mathbf v(\alpha_2) \right )}{m^2} 
\]
for some vectors $\mathbf v(\alpha_1), \mathbf v(\alpha_2) \in \mathbb Z^{\varphi(m)}$ and a diagonal positive integer matrix $D_m$ satisfying $q_1 = \mathbf v(\alpha_1)^\transpose D_m \mathbf v(\alpha_1)$ and $q_2 = \mathbf v(\alpha_2)^\transpose D_m \mathbf v(\alpha_2)$.

\paragraph{Large coherent rank.} 
In this paper, our focus has largely been on strictly Neumaier graphs that have the smallest possible coherent rank.
However, it is also natural to consider how large a Neumaier graph's coherent rank can be.

\begin{question}
    \label{prob:highRank}
    For each integer $r \geqslant 6$ does there exist a Neumaier graph having coherent rank $r$?
\end{question}

The Neumaier graphs $F_{(\pi_2,\dots,\pi_t)}(\operatorname{Cay}(\mathbb Z/pq\mathbb Z,\mathsf K_{0}(\alpha)))$ from Theorem~\ref{thm:adbkwz} provide promising candidates to answer Question~\ref{prob:highRank} in the affirmative.
Indeed, Abiad et al.~\cite{abiad2023infinite} have shown that there exist infinitely many such graphs or the form $F_{(\pi_2,\dots,\pi_t)}(\operatorname{Cay}(\mathbb Z/pq\mathbb Z,\mathsf K_{0}(\alpha)))$ when $q \equiv 1 \pmod 6$.
Based on cursory empirical evidence, a conservative lower bound for the coherent rank of such graphs appears to be $q+2$.
One expects even stronger lower bounds when $t > 1$.

Alternatively, the Neumaier graphs $\Gamma_m(\alpha_1,\alpha_2)$ also provide candidates by taking $m$ to be arbitrarily large.
However, in light of Theorem~\ref{thm:main}, we face two associated challenges.
Firstly, finding pairs of prime powers $q_1 = n_1m+1$ and $q_2=n_2m+1$ that satisfy \eqref{eqn:edgeregular}.
Secondly, verifying a sufficient condition, such as the one in Lemma~\ref{lem:fullRank}, to ensure a large coherent rank.

\paragraph{Neumaier graphs of each nexus.}

In Table~\ref{tab:list}, there is at least one example of a Neumaier graph whose nexus is equal to 1, 2, 4, 6, 10, and 16.
In Table~\ref{tab:list7}, there is at least one example of a Neumaier graph whose nexus is equal to 1, 4, 6, 7, 10, and 15.
Before this work, the literature only contained examples of strictly Neumaier graphs whose nexus is a power of $2$~\cite{EGP19}.
\begin{table}[h!tbp]
    \centering
    \begin{tabular}{c|l }
        Nexus &  $(m; q_1,q_2)$  \\
        \hline
         \multirow{2}{*}{1} & $(3;4,7)$, $(3;4,13)$, $(4;5,13)$, $(4;5,37)$, $(6;7,79)$, $(6;7,103)$, $(7;8,29)$, $(7;8,43)$, \\
         & $(7;8,71)$, $(7;8,127)$, $(10;11,131)$  \\
         \hline
         2 & $(3;7,16)$, $(3;7,19)$, $(5;11,41)$, $(5;11,101)$, $(6;13,37)$, $(9;19,487)$\\
         \hline
         3 & $(5;16,31)$, $(5;16,61)$, $(5;16,121)$, $(8;25,313)$, $(10;31,311)$, $(10;31,631)$  \\
         \hline
         \multirow{2}{*}{4} & $(3;13,16)$, $(3;13,49)$, $(4;17,25)$, $(7;29,113)$, $(7;29,449)$, $(9;37,181)$, \\
         & $(9;37,1171)$, $(10;41,401)$, $(10;41,601)$, $(10;41,1481)$ \\
         \hline
         5 & $(6;31,127)$, $(8;41,137)$  \\
         \hline
         \multirow{2}{*}{6} & $(3;19,67)$, $(4;25,97)$, $(5;31,181)$, $(5;31,211)$, $(5;31,256)$, $(8;49,337)$, $(10;61,701)$, \\
         & $(10;61,2221)$\\
         \hline
         \multirow{3}{*}{7} & $(4;29,61)$, $(4;29,229)$, $(6;43,691)$, $(9;64,199)$, $(9;64,307)$, $(9;64,343)$, $(9;64,613)$, \\
         & $(9;64,631)$, $(9;64,739)$, $(9;64,829)$, $(9;64,991)$, $(9;64,1009)$, $(9;64,1063)$, \\ & $(9;64,1153)$, $(9;64,2197)$ \\
         \hline
         8 &  $(5;41,71)$, $(5;41,131)$, $(6;49,73)$, $(6;49,1201)$, $(9;73,1621)$ \\
         \hline
         9 &   $(7;64,757)$, $(7;64,883)$, $(7;64,1583)$\\
         \hline
         10 & $(3;31,43)$, $(4;41,169)$, $(4;41,241)$, $(6;61,349)$, $(6;61,1237)$  \\
         \hline
         11 &  $(6;67,139)$ \\
         \hline
         12 & $(3;37,73)$, $(5;61,211)$, $(5;61,256)$, $(5;61,331)$, $(5;61,421)$, $(6;73,673)$, $(9;109,739)$ \\
         \hline
         13 &  $(6;79,1879)$, $(10;131,691)$, $(10;131,1091)$\\
         \hline
         14 &  $(5;71,281)$, $(5;71,461)$, $(9;127,397)$ \\
         \hline
         15 & $(4;61,349)$, $(10;151,431)$ \\
         \hline
         16 & $(3;49,193)$, $(6;97,1249)$  \\
         \hline
		 17 & $(6;103,2503)$, $(8;137,953)$ \\
         \hline
		 18 & $(10;181,1061)$\\
         \hline
		 19 & $(10;191,271)$, $(10;191,751)$\\
         \hline
		 20 & $(5;101,311)$, $(3;61,97)$, $(9;181,2161)$ \\
    \end{tabular}
    \caption{Table indicating the existence strictly Neumaier graphs $\Gamma_m(\alpha_1,\alpha_2)$ that have a given nexus at most $20$ for $m \leqslant 10$ and $q_2 \leqslant 5000$.}
    \label{tab:listNexus}
\end{table}

For suitably chosen primitive elements $\alpha_1$ and $\alpha_2$ of $\operatorname{GF}(16)$ and $\operatorname{GF}(31)$, respectively, one can check that $\Gamma_5(\alpha_1,\alpha_2)$ is a strictly Neumaier graph with nexus equal to $3$.
Furthermore, in Table~\ref{tab:listNexus}, we list Neumaier graphs $\Gamma_m(\alpha_1,\alpha_2)$ that have a given nexus for $m \leqslant 10$ and nexus at most 20.
In the table, we list the triple $(m;q_1,q_2)$ which corresponds to the graph $\Gamma_m(\alpha_1,\alpha_2)$ for suitably chosen primitive elements $\alpha_1$ and $\alpha_2$ of $\operatorname{GF}(q_1)$ and $\operatorname{GF}(q_2)$, respectively.

Based on Table~\ref{tab:listNexus}, since for each $e \in \{1,\dots,20\}$, we are able to find a Neumaier graph with nexus $e$, it is natural to ask the following question.

\begin{question}
    \label{prob:Nexus}
    For each integer $e \geqslant 1$, does there exist a strictly Neumaier graph having nexus equal to $e$?
\end{question}
 \section{Acknowledgements}

The first author extends his gratitude to Misha Klin, who inspired him to look at the coherent closure of the Neumaier graphs constructed in \cite{GreavesKoolen1}.
The authors have also benefitted from conversations with Jack Koolen, Ka Hin Leung, and Bernhard Schmidt. 
The first author was partially supported by the Singapore Ministry of Education Academic Research Fund; grant numbers: RG14/24 (Tier 1) and MOE-T2EP20222-0005 (Tier 2).

\bibliographystyle{amsplain}
\bibliography{bibliography}

\end{document}